\documentclass{amsart}
\usepackage{amsmath}
\usepackage{amssymb}
\usepackage[all]{xy}
\usepackage{comment}
\usepackage{color}

 \usepackage[notcite,notref]{showkeys}

\theoremstyle{plain}
\newtheorem{thm}{Theorem}[section]
\newtheorem{thm*}{Theorem}[section]
\newtheorem{cor}[thm]{Corollary}
\newtheorem{prop}[thm]{Proposition}

\newtheorem{lemma*}{Lemma}

\theoremstyle{definition}
\newtheorem{defn}[thm]{Definition}
\newtheorem{remark}[thm]{Remark}

\newtheorem*{remark*}{Remark}

\newtheorem{ex}[thm]{Example}

\newtheorem{question*}{Question}

\numberwithin{equation}{thm}

\newcommand{\cN}{\mathcal N}

\def\Spec{\operatorname{Spec}\nolimits}

\newcommand{\bG}{\mathbb G}

\newcommand{\bA}{\mathbb A}

\newcommand{\bF}{\mathbb F}

\newcommand{\cC}{\mathcal C}

\newcommand{\cE}{\mathcal E}

\newcommand{\fg}{\mathfrak g}
\newcommand{\fh}{\mathfrak h}

\newcommand{\fu}{\mathfrak u}

\newcommand{\gl} {\mathfrak {gl}}

\newcommand{\ol}{\overline}
\newcommand{\ul}{\underline}

\def\Spec{\operatorname{Spec}\nolimits}
\def\sl2{\operatorname{SL_{2(2)}}\nolimits}
\def\Ga2{\operatorname{\mathbb G_{\rm a(2)}}\nolimits}

\newcommand{\Z}{\mathbb Z}

\newcommand{\E}{\mathcal E}

\date\today

\begin{document}

 \title[Support varieties for rational representations]{Support varieties for rational representations}
 
 \author[ Eric M. Friedlander]
{Eric M. Friedlander$^{*}$} 

\address {Department of Mathematics, University of Southern California,
Los Angeles, CA}
\email{ericmf@usc.edu}
\email{eric@math.northwestern.edu}

\thanks{$^{*}$ partially supported by the NSF grants DMS-0909314 and DMS-0966589.}

\subjclass[2010]{20G05, 20C20, 20G10}

\keywords{support variety, 1-parameter subgroup}

\begin{abstract}   
We introduce support varieties for rational representations of a linear
algebraic group $G$ of exponential type over an algebraically closed field $k$
of characteristic $p > 0$.  These varieties are closed subspaces of
the space $V(G)$ of all 1-parameter subgroups of $G$.  The functor $M \mapsto
V(G)_M$ satisfies many of the standard properties of support varieties 
satisfied by 
finite groups and other finite group schemes.  Furthermore, there is a close 
relationship between $V(G)_M$ and the family of support varieties $V_r(G)_M$
obtained by restricting the $G$ action to Frobenius kernels $G_{(r)} \subset G$.
These support varieties seem particularly appropriate for the investigation of
 infinite dimensional rational $G$-modules.
\end{abstract}

\maketitle


\section{Introduction}

The purpose of this paper is to formulate  a suitable theory of support varieties for rational 
representations for a natural class of linear algebraic groups $G$ which includes the classical 
simple groups.    We work over an algebraically closed field $k$ of characteristic $p > 0$, so
that we consider modular representations of $G$: actions of $G$ on $k$-vector spaces.
Our criteria for ``suitability" include i.) a description which reflects the structure
of $G$; ii.) a theory that applies to all rational representations $M$ of 
$G$;  iii.) expected properties for direct sums, tensor products, extensions, and Frobenius twists;
and iv.) a structure $V(G)_M$ which incorporates the information of the support 
variety of the rational representation $M$ of $G$ when restricted to any Frobenius kernel
$G_{(r)} \subset G$.  Our formulation is an extension of the approach 
 of C. Bendel, A. Suslin, and the author  \cite{SFB2}; we employ 1-parameter subgroups 
rather than traditional methods of cohomology
(e.g., \cite{Ca}) or the more recent methods of $\pi$-points (e.g., \cite{FP1}).    The reader is referred to
 \cite{F2} for a brief history of support varieties, beginning with the fundamental work of 
 D. Quillen \cite{Q1}, \cite{Q2}.  For brevity, we usually use ``rational $G$-module" to refer to a rational
 representation of $G$.

We remind the reader that support varieties (for representations of a finite group, a restricted Lie algebra,
or the infinitesimal kernel of a linear algebraic group) give some measure of the local projectivity of the 
representation with respect to $p$-nilpotent actions.  Support varieties have their origins in the formulation of
the complexity (rate of growth) of projective resolutions and they reflect properties of extensions 
rather than structures of irreducibles.   Refined support varieties have been introduced by the author
and J. Pevtsova \cite{FP2} in order to capture further information about representations of finite group schemes.
We extend these invariants to rational representations of linear algebraic groups of exponential type.  

The existence of an appropriate theory of support varieties for rational representations
is not evident.  For example, the rational 
cohomology of $G$ with coefficients in $k$ ususally vanishes in positive degrees for simple algebraic groups over $k$
so that the familiar cohomological methods do not apply.   Indeed, projective resolutions of rational
representations typically do not exist, and injective representations are typically infinite dimensional.
In fact, we have no direct cohomological interpretation of our invariants for a linear algebraic group $G$,
other than their relationship with invariants for the Frobenius kernels $G_{(r)}$ of $G$.
Furthermore, the formulation of local projectivity for rational $G$-modules
appears unpromising at first for there does not
appear to be a good space of $p$-nilpotent operators on which to test local projectivity.
Another obstacle which arises in formulating a theory of support varieties for rational $G$-modules is 
that invariants for the infinitesimal kernels $G_{(r)}$ of a linear algebraic
group $G$ do not ``match up" with respect to either the natural projections $G_{(r+1)} \to G_{(r)}$ or the 
natural embeddings $G_{(r)} \to G_{(r+1)}$.

Our formulation of the support variety $V(G)_M$ (Definition \ref{supportG}) 
of a rational $G$-module $M$ is as a subset of $V(G)$, the set of all 1-parameter subgroups of $G$
given the induced topology as a natural subset of  the inverse limit of the $k$-rational points of 
schemes of  infinitesimal 1-parameter subgroups of $G$; for $M$ finite dimensional, $V(G)_M$ is
closed in $V(G)$.  This formulation entails the definition of the ($p$-nilpotent) action of $G$ on $M$ 
at a 1-parameter subgroup $\bG_a \to G$ of $G$.  (See Definition \ref{Gaction}.)   A somewhat
confusing twist of indexing is required to enable compatibility of this formulation with that 
for the support variety of $M$ restricted to Frobenius kernels $G_{(r)} \subset G$ of G.  

Extending work of J. Pevtsova and the author in \cite{FP2} for finite group schemes, 
we  introduce for each $j > 0$ the``non-maximal $j$-rank 
varieties" $V^j(G)_M$ for a finite dimensional rational $G$-module $M$ which 
detect further information about $M$ given in terms of the locus
of  Jordan types.  (See Definition \ref{defn:gen}.)  The key ingredient of this refinement is the 
well-definedness of  maximal Jordan types for a finite group scheme proven in \cite{FPS}.

The linear algebraic groups for which we construct a theory of support varieties are those with a 
structure of exponential type as formulated in Definition \ref{str-exptype}.  
Work of P. Sobaje shows that a reductive group $G$ has such a structure provided
that $p \geq h(G)$,  the Coxeter number of $G$ \cite{S2}.  Other  examples of such groups
are simple groups of classical type, parabolic subgroups of such groups,
and unipotent radicals of these parabolic subgroups \cite{SFB1}.  

Basic properties satisfied by $M \mapsto V(G)_M$ are given in Theorem \ref{Ga-items}.
These include the expected behavior with respect to direct sums, tensor products, and Frobenius twists
of rational $G$-modules.  If $M$ is finite dimensional, then $V(G)_M \subset V(G)$ is a closed, $G$-stable
subset; moreover, for $M$ finite dimensional,  $V(G)_M$  is determined
by the restriction of $M$ to $G_{(r)}$ for $r$ sufficiently large (depending upon $M$).

For a given finite dimensional rational $G$-module $M$, the computation of the support variety of $M$ reduces to a 
computation of support variety of $M$ as a $G_{(r)}$-module for $r$ sufficiently large depending upon 
$M$ (see Theorem \ref{G-items}).  The ``Jantzen Conjecture"  (proved in  \cite{NPV}) and \cite{O}) enables 
some computations in Proposition \ref{induced-tensor}.   
Examples \ref{ex-St} and \ref{ex-max} are explicit computations of certain $V(G)_M$ and $V^j(G)_M$.

In the final section of this paper, we consider natural examples of infinite dimensional rational $G$-modules.
In Proposition \ref{ker-inj}, we should that if $G$ admits a structure of exponential type and $L$ is an
injective rational $G$-module, then the support variety $V(G)_L$ is trivial.  This, together with general 
properties of our support varieties, leads to various examples given at the end of Section \ref{sec:infdim}.

Throughout this paper, we consider affine group schemes of finite type over an algebraically
closed field $k$ of characteristic $p > 0$.  We use the terminology ``linear algebraic group" to mean
a reduced, irreducible group scheme of finite type over $k$ which admits a closed embedding 
into some general linear group $GL_n$ over $k$.  We refer the reader to \cite{J}, \cite{SFB1}, and
\cite{SFB2} for some general background we require.

We express our gratitude to Julia Pevtsova for numerous conversations on matters related to
support varieties and to Paul Sobaje for various results found in \cite{S1} and \cite{S2} which
have influenced our thinking about 1-parameter subgroups and shaped our definition of a
linear algebraic group of exponential type. 


\section{1-parameter subgroups}

In this first section, we begin by recalling a few facts about (infinitesimal) 1-parameter subgroups
$\bG_a \to G$ of $G$, where $\bG_a$ is the additive group.  
We  explore the special  case $G = \bG_a$, recalling a concrete description of all 1-parameter subgroups 
of $\bG_a$.   Thie example of $G = \bG_a$ will serve as a guide for much more general $G$, those of 
exponential type.
The analogous form of 1-parameter subgroups for classical groups motivates our formulation 
in Definition  \ref{str-exptype} of a  linear algebraic group with a structure of exponential type.  
These are the algebraic
groups for which we construct support varieties.  We conclude this section with a determination of the
effect of pre-composition and post-composition with the Frobenius morphism.

For any group scheme $G$ of finite type over $k$ and and $r > 0$, we denote by $F^r: G \to G^{(r)}$ 
the $r$-th Frobenius morphism, where $G^{(r)}$ is the base change of $G$ along the $p^r$-th power 
map $\phi: k \to k$ (see\cite[\S 1]{FS}); 
we denote by $G_{(r)}$ the $r$-th Frobenius kernel of $G$: 
$$G_{(r)} \ \equiv \ ker\{ F^r: G \to G^{(r)} \}.$$
When $G$ is defined over $\bF_{p^r}$ (i.e., $G = \Spec k \times_{\Spec \bF_{p^r}} G_0$ for some group scheme
$G_0$ defined over $\bF_{p^r}$), then we shall use the natural identification of $G^{(r)}$ with $G$.

\begin{defn}
Let $G$ be a connected group scheme of finite type over $k$.  

For any $r > 0$, we denote by $V_r(G)$ the affine scheme (cf. \cite{SFB1})
$$V_r(G)  \ \equiv \ Hom_{grp/k}(\bG_{a(r)},G)$$
of homomorphisms of group schemes over $k$ from the $r$-th Frobenius kernel of the additive group 
to  $G$.  Such a homomorphism will be called an infinitesimal
1-parameter subgroup of $G$ (of height $r$).  The set of $k$-points of the scheme $V_r(G)$, $V_r(G)(k)$, 
is endowed with the Zariski topology.

We denote by $V(G)$ the topological subspace
\begin{equation}
\label{V(G)}
V(G) \ \equiv \ Hom_{grp/k}(\bG_a,G)  \ \subset \ \varprojlim_r V_r(G)(k)
\end{equation} 
of 1-parameter subgroups of $G$ viewed as a subspace of $\varprojlim_r V_r(G)(k)$
endowed with the inverse limit topology.  Thus, the topology on $V(G)$ is
the weakest topology making  each projection map
$pr_r: V(G) \to V_r(G)(k)$ continuous, where $pr_r$ is defined by sending $\psi:\bG_a \to G$ to
its restriction to $\bG_{a(r)}$.
\end{defn}

Before giving examples, we mention some useful properties of $V_r(G)$.

\begin{prop}\label{1param}
Let $G$ be a linear algebraic group and $G \subset GL_n$ a closed embedding.
\begin{enumerate}
\item
$V_r(G) \ \simeq \ Hom_{grp}(\bG_{a(r)},G_{(r)})$.
\item
$G_{(r)} \ = \ G \ \cap \ GL_{n(r)}$.
\item
$V_r(G)  \ = \ V(G) \cap V_r(GL_n)$.
\item
$G$ acts (via conjugation)
on each of the schemes $V_r(G)$ and $G(k)$ acts (via conjugation) on the space $V(G)$.
\end{enumerate}
\end{prop}

\begin{proof}
Assertion (1) follows from the observation that any map of group schemes 
$\bG_{a(r)} \to G$ factors uniquely through the closed embedding $i_r: G_{(r)} \to G$.

Assertion (2) can be found in \cite[I.9.4]{J}, and assertion (3) follows immediately
from (1) and (2).

Assertion (4) is easily verified by viewing the action of $G$ on $V_r(G)$ as a
functorial action of $G(A)$ on $(V_r(G))(A)$ as $A$ ranges overs
 finitely generated commutative $k$-algebras.
\end{proof}

Our first example is $G = \bG_a$, the additive group.  Throughout this paper, the 
special case $G = \bG_a$ will serve as our ``test case" for new constructions.

\begin{ex}
\label{additive}
We denote by 
$$\sigma_{\ul a} \ \equiv \ \sum_{s \geq 0} a_sF^s: \bG_a \to \bG_a$$
the 1-parameter subgroup of $\bG_a$ given by a finite sequence 
$\ul a = (a_0,\ldots,a_s,\ldots ) \in \bA^\infty$ (i.e., $a_s = 0$ for $s > > 0$) of elements of $k$.
Here, $F^s: \bG_a \to \bG_a$ is the $s$-th iterate of the Frobenius morphism
given by the $p^s$-th power map $k[T] \to k[T], \ T \mapsto T^{p^s}$. 
 Thus,
the map on coordinate algebras $\sigma_{\ul a}^*$ is given by
$T\mapsto \sum_{s \geq 0} a_sT^{p^s}$.  In the special case that $\ul a$
has the single non-zero term $a_0 = b$, the 1-parameter subgroup 
$\sigma_{\ul a}: \bG_a \to \bG_a$ is
just multiplication by $b$ and will be denoted $\sigma_b$.   For any $r >0$, 
there is a natural isomorphism of schemes \cite[1.10]{SFB1}
$$ \bA^r \ \stackrel{\sim}{\to} \ V_r(\bG_a) , \quad \ul a \mapsto
\sigma_{\ul a} \circ i_r\ = \  (\sum_{s\geq 0}^{r-1} a_sF^s) \circ i_r: \bG_{a(r)} \to \bG_a \to \bG_a.$$
\end{ex}

\vskip .1in

For a $p$-nilpotent, $n\times n$ matrix $A$ with entries in $k$, we denote by  
$$exp_A: \bG_a \to SL_n \subset GL_n$$
the 1-parameter subgroup  given by the functor on commutative $k$-algebras $R$ sending
$r \in R$ to the matrix $\sum_{s=0}^{p-1} r^s A^s/(s!).$
Let $\gl_n$ denote the (restricted) Lie algebra of $GL_n$ and let $\cN_p(\gl_n) \subset \gl_n$ 
denote the closed subvariety $p$-nilpotent matrices.
We denote by 
\begin{equation}
\label{crN}
C_r(\cN_p(\gl_n)) \ \subset \ (\cN_p(\gl_n))^{\times r}
\end{equation} the variety of $r$-tuples $(A_0,\ldots,A_{r-1})$ of $p$-nilpotent,
pair-wise commuting $n\times n$ matrices.  We let $C_\infty(\cN_p(\gl_n))$ denote
the colimit (i.e., union) of the $C_r(\cN_p(\gl_n))$, so that a point
of  $C_\infty(\cN_p(\gl_n))$ is a finite sequence $(A_0,\ldots,A_s,\ldots)$ (i.e., $A_s = 0$ for 
$s > > 0$) of $p$-nilpotent, 
pair-wise commuting $n\times n$ matrices.

Example \ref{additive} has the following analogue for $G = GL_n$.

\begin{ex}
\label{linear}
For any $\ul A = (A_0,A_1,\ldots,) \in C_\infty(\cN_p(\gl_n))$,  we denote by 
$$exp_{\ul A} \ \equiv \  \prod_{s\geq 0} (exp_{A_s} \circ F^s): \bG_a \to GL_n
$$
the indicated 1-parameter subgroup of $GL_n$.  For any $r >0$, 
\begin{equation}
 \label{reverse}
\ul A \ \mapsto \ exp_{\ul A} \circ i_r: \bG_{a(r)} \to \bG_a \to GL_n
\end{equation}
determines a natural isomorphism \cite[1.4]{SFB1}
$$C_r(\cN_p(\gl_n)) \ \stackrel{\sim}{\to} \ V_r(GL_n).$$
\end{ex}

As we discuss in the next example, Example \ref{linear} extends to linear algebraic
groups $G$ with an embedding of exponential type $G \subset GL_N$.  We recall that  
simple algebraic groups of classical type with their natural embeddings into linear groups are embeddings 
of exponential type \cite[1.8]{SFB1}.   A particularly simple class of examples of embeddings
of exponential type are the embeddings of  root subgroups $\bG_a \subset G$
of a reductive group $G$ as in Example \ref{Xij}.

\begin{ex}
\label{exptype}
A closed embedding $j: G \subset GL_N$ of  algebraic groups  is said 
to be of {\it exponential type} (as in \cite{SFB1}) if
for all commutative $k$ algebras $A$ and all $p$-nilpotent $x \in \fg\otimes A$ the exponential map
$$exp_{(dj)(x)}: \bG_a(A) \ \to \ GL_N (A)$$
factors through $j: G(A) \to GL_N(A)$.  Here, $\fg$ is the Lie algebra of $G$.
This condition is equivalent to the condition that the exponential map for $GL_N$
restricts to determine a map of $k$-schemes
$$\E: \cN_p(\fg) \times \bG_a \ \to \ G.$$

For such $j: G \subset GL_N$ of exponential type, (\ref{reverse}) restricts to isomorphisms
$$C_r(\cN_p(\fg))  \ \stackrel{\sim}{\to} \  V_r(G), \quad \ul B \mapsto \  exp_{\ul B} \circ i_r =
\prod_{s=0}^{r-1}(exp_{B_s} \circ F^s) \circ i_r: \bG_{a(r)} \to \bG_a \to G.$$
\end{ex}

We formulate a class of linear algebraic groups  $G$ more general than those admitting a closed embedding
of exponential type.  These are the algebraic groups $G$ for which we construct our theory of 
support varieties.

\begin{defn}
\label{str-exptype}
Let $G$ be a linear algebraic group with Lie algebra $\fg$.   A structure of exponential type
on $G$ is a $G$-equivariant morphism of $k$-schemes (with respect to the adjoint actions 
on $\fg$ and $G$)
\begin{equation}
\label{Exp}
\cE: \cN_p(\fg) \times \bG_a \ \to G, \quad (B,s) \mapsto \cE_B(s)
\end{equation}
such that
\begin{enumerate}
\item
For each $B\in \cN_p(\fg)(k)$, $\cE_B: \bG_a \to G$ is a 1-parameter subgroup.
\item
For any pair of  commuting $p$-nilpotent elements $B, B^\prime \in \fg$,
the maps $\cE_B, \cE_{B^\prime}: \bG_a \to G$ commute.
\item
For any commutative $k$-algebra $A$, any $\alpha \in A$,  and any 
$s \in \bG_a(A)$, \ $\cE_{\alpha \cdot B}(s) = \cE_B(\alpha\cdot s)$.
\item
Every 1-parameter subgroup $\psi: \bG_a \to G$ of $G$  is of the form 
\begin{equation}
\label{exponent}
\cE_{\ul B} \ \equiv \  \prod_{s=0}^{r-1} (\cE_{B_s} \circ F^s)
\end{equation}
for some $r > 0$, some $\ul B \in C_r(\cN_p(\fg))$;
furthermore,  $C_r(\cN_p(\fg)) \ \to \ V_r(G),\quad \ul B \mapsto \cE_{\ul B} \circ i_r$ 
is an isomorphism.
\end{enumerate}
\end{defn}

\vskip .1in

Condition (2) of Definition \ref{str-exptype} is equivalent to the condition that the map 
$\cE_B \bullet \cE_{B^\prime}: \bG_a \to G\times G$ factors as a map of group schemes
 through the diagonal map
$diag: \bG_a \to \bG_a \times \bG_a$.

Observe that the condition on $G$ that it should admit a structure of exponential type implies that
every infinitesimal 1-parameter subgroup $\bG_{a(r)} \to G$ admits a natural 
lifting to a 1-parameter subgroup $\bG_a \to G$.  
Furthermore, if $\psi: \bG_a \to G$ satisfies the condition that  $\psi^*$ applied to each element 
of some set of generators of $k[G]$  is a polynomial 
in $k[\bG_a] = k[T]$ of degree $< p^r$, then $\psi = \cE_{\ul B}$ for some $\ul B \in C_r(\cN_p(\fg))$.

\begin{remark}
By Definition \ref{str-exptype}(4), a structure $\cE: \cN_p(\fg) \times 
\bG_a \to G$ of exponential type determines a ``$p$-nilpotent Springer 
isomorphism" $\cN_p(\fg) \stackrel{\sim}{\to} U_p(G), B \mapsto \cE_B(1)$. 
 If $\cE^\prime$ is another structure of exponential 
type on $G$, then there is a unique automorphism $\phi_{\cE,\cE^\prime}: 
\cN_p(\fg) \to \cN_p(\fg)$ relating the $p$-nilpotent Springer isomorphisms
associated to $\cE, \cE^\prime$.  Moreover, $\phi_{\cE,\cE^\prime}$
determines an automorphism $\Phi_{\cE,\cE^\prime}: V(G) \to V(G)$ using (\ref{exponent}).
As is readily verified, the support variety $V(G)_M \subset V(G)$ of a rational 
$G$-module $M$ as defined in Definition \ref{supportG} with respect to $\cE$ 
has image under $\Phi_{\cE,\cE^\prime}$ the support variety of $M$ with 
respect to $\cE^\prime$.
\end{remark}

\begin{ex}
\label{Sobaj}
As shown by Sobaje in \cite[4.2]{S2}, if $G$ is reductive with $p \geq h$, and if $PSL_p$ is not a 
factor of $[G,G]$, then $G$ can be given a structure of exponential type by defining $\cE$ 
to be the exponential map
constructed by Seitz in \cite[5.3]{Sei} on a Borel subgroup  and extending to a $G$-equivariant 
map using work of  Carlson, Lin, Nakano \cite{CLN} and McNinch \cite{McN}.  
Moreover, if $G$ is such a reductive group
and if $H \subset G$ is a parabolic subgroup  or  the unipotent radical in a parabolic subgroup
and if $\fh = Lie(H)$,
then the restriction of $\cE$ to $\cN_p(\fh)$ provides $H$ with a structure of exponential type.

Recent work of Sobaje suggests that any reductive group $G$ can be given a structure of 
exponential type provided that $p$ is very good for $G$.
\end{ex}

We provide here the evident definition of a map of linear algebraic groups equipped with structures
of exponential type.  A natural example of such a map is a closed embedding $G \subset GL_N$ 
of exponential type as in Example \ref{exptype}.   If $G, \ G^\prime$ are provided with structures $\cE,\ \cE^\prime$
of exponential type, then the inclusion $1\times e: G \to G\times G^\prime$ 
(sending $g \in G$ to $(g,e) \in G\times G^\prime$) and the projection
$pr_1: G\times G^\prime \to G$ are maps of exponential type.

\begin{defn}
\label{defn-exp}
Let $G, \ G^\prime$ be linear algebraic groups equipped with structures of exponential type $\cE, \ \cE^\prime$.  
Then a homomorphism of algebraic groups $f: G \to G^\prime$ is said to be a map of exponential type
if the following square commutes:
\begin{equation}
\label{diffl}
\begin{xy}*!C\xybox{%
\xymatrix{
\cN_p(\fg) \times \bG_a  \ar[rr]^{\cE} \ar[d]_{df \times id}  &   & G \ar[d]^f\\
\cN_p(\fg^\prime)\times  \bG_a  \ar[rr]^{\cE^\prime} & &  G^\prime }
}\end{xy}
\end{equation}
\end{defn}

The following Example \ref{ex:sep} includes parabolic subgroups of
reductive groups (which are  semi-direct products of their unipotent radicals and  Levi quotients).

\begin{ex}
\label{ex:sep}
Let $G$ be a linear algebraic group equipped with a structure of exponential type and assume
that $G$ can be writen as the semi-direct product $G \ \simeq \ H \rtimes K$;
$H \subset G$ is a reduced, closed subgroup and $\pi: G \to K$ admits a splitting $s: K \to G$. 
If both $H, s(K) \ \subset \ G$ are embeddings of exponential type, then
$\pi: G \to G/K$ is also a map of exponential type.
\end{ex}

We conclude this section by making explicit the actions on $V(G)$ of pre-composition and post-composition
with the Frobenius morphism.

\begin{prop}
\label{twist}
Let $G$ be a linear algebraic group equipped with a structure of exponential
type.  Then precomposition with  the Frobenius morphism induces the self-map
$$(-\circ F): V(G) \to V(G), \quad \cE_{\ul B} \mapsto \cE_{\ul B \circ F}$$
where $(B_0,B_1,\ldots) \circ F \ = \  (0,B_0,B_1\ldots).$

Furthermore, if $G \hookrightarrow GL_n$ is an embedding of exponential type defined over $\bF_p$, then 
post-composition with the Frobenius morphism induces the self-map
$$(F \circ -): V(G) \to V(G), \quad exp_{\ul B} \mapsto \exp_{F( \ul B) }$$
where $F (B_0,B_1,\ldots) \ = \  (0, B_0^{(1)},B_1^{(1)},\ldots).$
Here, $B^{(1)}$ is the $n\times n$-matrix obtained by raising each entry of $B$
to the $p$-th power.
\end{prop}

 \begin{proof}
 The identification of $(-\circ F)$ is immediate from the definition of $\cE_{\ul B}$.
 
 Assume now that $G \hookrightarrow GL_n$ is an embedding of exponential type defined over $\bF_p$.
 In particular, $G$ is defined over $\bF_p$ so that we may view the Frobenius morphism as a self-map
 $F: G\to G$ which thus induces $(F \circ -): V(G) \to V(G)$ which in turn is identified
 as the restriction of $F: V(GL_n) \to V(GL_n)$.    
 
We verify that 
$$F \circ (exp_B \circ F^s)  \ = \ exp_{B^{(1)}} \circ F^{s+1} $$
 by establishing this equality as an equality of functors on  commutative $k$-algebra $A$:
 for any $a \in A$,  
$$F( (exp_B \circ F^s)(a)) \ = \ \big(1 + B\cdot a^{p^s} + B^2 \cdot 
a^{2p^s}/2 + \cdots + B^{p-1} \cdot a^{(p-1)p^s}/(p-1)!\big)^{(1)}$$
equals 
$$(exp_{B^{(1)}}\circ F^{s+1})(a) \ = \ 1 + B^{(1)} \cdot a^{p^{s+1}} + (B^{(1)})^2 \cdot a^{2p^{s+1}}/2 + \cdots + 
(B^{(1)})^{p-1} \cdot a^{(p-1)p^{s+1}}/(p-1)!.$$
Consequently, 
$$F \circ \exp_{\ul B} \circ F^s \ \equiv \ F \circ (\prod_{s\geq 0} exp_{B_s}\circ F^s) \ = \ \prod_{s \geq 0} exp_{B_s^{(1)}} \circ F^{s+1}\
\equiv  \ exp_{F(\ul B)} .$$
 \end{proof}
 
 \begin{cor}
As in Proposition \ref{twist}, let  $G$ be a linear algebraic group equipped with a structure of exponential
type.  Then $\cE_{\ul B} \in pr_r^{-1}(\{ 0 \})$ if and only if $\ul B = \ul B^\prime \circ F^r$ for some $\ul B^\prime
\in \cC_\infty(\cN_p(\fg))$.
 \end{cor}


\section{Action on rational modules at 1-parameter subgroups}

In this section, we begin consideration of rational $G$-modules and the role 1-parameter
subgroups of $G$ plays in determining their structure.   Much of this section is 
directed to presenting and justifying the formulation of the action of $G$ on a rational
$G$-module at a 1-parameter group (Definition \ref{Gaction})
for $G$ a linear algebraic group of exponential type.  Because this definition
might seem somewhat opaque at first, we treat first the case that $G = \bG_a$.   One 
important observation is Proposition \ref{bound} which asserts that for a given 
rational $G$-module $M$ there is some integer $r$ independent of the choice of 
1-parameter subgroup of $G$ such that this action only 
depends upon the first $r$ terms of the sequence defining a 1-parameter subgroup.  
We conclude this section with a brief investigation of the ``group algebra" of the 
linear group $GL_n$.

We shall find it convenient to have at hand various equivalent formulations of 
the structure of a rational $G$-module as discussed in Proposition \ref{act}.
Such a structure determines the structure given in (\ref{rat-act})
of a module over the ``group algebra" $kG$ of $G$
(also referred to as the ``hyperalgebra"  of $G$ as in  \cite{CPS} or the
``algebra of distributions" at the identify of $G$ and  denoted $Dist(G)$ as in  \cite{J},). 
 
We begin by recalling the definition of $kG$.

\begin{defn}
Let $G$ be a connected affine group scheme of finite type over $k$.  
Denote by $kG_{(r)}$ the finite dimensional $k$-algebra defined as the $k$-linear 
dual of $k[G_{(r)}]$ whose product is determined by the coproduct structure on $k[G]$.
The group algebra $kG$ of $G$ is the $k$-algebra
$$kG \ \equiv \ \varinjlim_r kG_{(r)}. $$
\end{defn}

\begin{prop}
\label{act}
Let $G$ be a connected affine group scheme of finite type over $k$ and $M$ a $k$-vector space.   
The following structures on $M$ are equivalent.
\begin{enumerate}
\item
The structure 
\begin{equation}
\label{DeltaM}
\Delta_M: M \ \to M\  \otimes k[G]
\end{equation}
 of a comodule for the 
coordinate algebra $k[G]$ of $G$.  
\item
The structure of  functorial (with respect to commutative
$k$-algebras $A$) $A$-linear group actions
\begin{equation}
\label{thetaM}
\Theta_{M,A}: G(A) \times (M\otimes A) \ \to (M\otimes A).
\end{equation}
\item
The structure of a $k[G]$-linear group action:
\begin{equation}
\label{spec}
\Theta_{k[G]}: G(k[G]) \times (M\otimes k[G]) \to (M\otimes k[G]).
\end{equation}
\item
For $M$ finite dimensional, equipped with a basis $\{ m_1,\ldots, m_n\}$, the
structure of a map 
\begin{equation}
\label{rhoM}
\rho_M: G \ \to \ GL_n
\end{equation} of group schemes (over $k$).
\end{enumerate}

The vector space $M$ equipped with one of the equivalent structures listed above is
said to be a rational representation of $G$, or (more briefly) a  {\it rational $G$-module}.
Such a structure determines a locally finite $kG$-structure on $M$ given by 
\begin{equation}
\label{rat-act}
kG \otimes M \to M, \quad (\phi \in kG, m \in M) \ \mapsto \ \sum_i \phi(f_i)m_i, \ {\text where} \
\Delta_M(m) = \sum_i m_i \otimes f_i.
\end{equation}
\end{prop}

\begin{proof}
Given the coproduct $\Delta_M$ as in (\ref{DeltaM}), the functorial pairings $\Theta_{M,A}$ (\ref{thetaM})  
are given by 
\begin{equation}
\label{pairing}
(x: k[G] \to A, m \otimes 1) \quad \mapsto \quad  (id_M \otimes x)(\Delta(m)).
\end{equation}
The pairing $\Theta_{k[G]}$ of  (\ref{spec}) is a special case of (\ref{thetaM}) with $A = k[G]$.  
On the other hand, the pairing $\Theta_{k[G]}$ of (\ref{spec}) determines  the pairing (\ref{DeltaM}) by
$$(\Delta_M)(m) \ = \ \Theta_{k[G]}(id \times m \otimes 1) \ \in \ M\otimes k[G].$$ 
In particular, the structures of (\ref{DeltaM}), (\ref{thetaM}), and (\ref{spec}) are equivalent.

The comodule structure on $M$ given by  (\ref{DeltaM}) determines the $kG$-action
on $M$ as made explicit in (\ref{rat-act}).   This $kG$-action on $M$ must be locally finite,
since the image under $\Delta_M$ of any finite 
dimensional subspace of $M$ is necessarily a finite dimensional subspace of $M \otimes k[G]$.
For $M$ finite dimensional equipped with a basis $\{ m_1,\ldots, m_n \}$, the adjoints of the 
pairings (\ref{thetaM}) are functorial (with respect to $A$) group homomorphisms
$G(A) \ \to \ Aut_A(M\otimes A) \simeq GL_n(A)$ which is equivalent to the structure 
$\rho_M: G \to GL_n$ of (\ref{rhoM}).

To complete the proof of the proposition, it suffices to assume that $M$ is finite dimensional and
show that the structure $\rho_M: G \to GL_n$ of (\ref{rhoM}) determines the coproduct $\Delta_M$
of (\ref{DeltaM}).   We do this by  identifying $M$ as a vector space with $V_n$
(the defining rational $GL_n$-module) and taking  $\Delta_M$ to be the 
composition $(1 \otimes \rho_M)^* \circ \Delta_{V_n}: V_n \to V_n \otimes k[GL_n] \to V_n \otimes k[G]$.
\end{proof}

For the example $G = \bG_a$, we proceed to identify explicitly the multiplication by elements of 
$k\bG_a$ on a rational $\bG_a$-module $M$. 

\begin{ex}
\label{ex-ga}
The group algebra $k\bG_a$ is given as
$$k\bG_a \ \equiv \ k[u_0,\ldots,u_i,\ldots]/(u_i^p) \ \subset \ k[T]^*$$
where $u_i$ applied to $f \in k[\bG_a] = k[T]$ reads off the coefficient of $T^{p^i}$.   
Here $(u_i^p)$ denotes the ideal of $k[u_0,\ldots,u_i,\ldots]$ generated by $\{ u_0^p, \ldots,u_i^p \ldots\}$.
On group algebras, we have  
$$i_{r*}: k\bG_{a(r)} = k[u_0,\ldots,u_{r-1}]/(u_i^p) \ \subset \ k[u_0,\ldots,u_n,\ldots]/(u_i^p) = k\bG_a,$$
\begin{equation}
\label{Fu}
F_*: k\bG_a \to k\bG_a, \quad  F_*(u_i) = u_{i-1}\ if \  i >0, \ F_*(u_0) = 0.
\end{equation}

Let  $\Delta_M: M \to M \otimes k[T]$ be the defining coaction of a rational $\bG_a$-module $M$.
For $m \in M$ and any $ v_j  \ =  \ \frac{u_0^{j_0} \cdots u_{r-1}^{j_{r-1}}}{j_0! \cdots j_{r-1}!}
\in k\bG_a$ (with $j = \sum_{\ell=0}^{r-1} j_\ell p^\ell, \ 0 \leq j_\ell < p$),
the action of $v_j \in k\bG_a = k[u_0,\ldots,u_i,\ldots]/(u_i^p)$ on $M$ is given by
\begin{equation}
\label{nabla}
\Delta_M(m) \ = \ \sum_{j\geq 0} v_j(m) \otimes T^j.
\end{equation}
Here,  the sum is finite for each $m \in M$.  If $M$ is finite dimensional, then 
$v_j(m)$ is non-zero only for finitely many values of $j$ as $m$ ranges over a 
basis for $M$; thus
$u_s$ acts trivially on $M$ for $s$ sufficiently large. The image of each $u_s$ in 
$End_k(M)$ is $p$-nilpotent, and the $u_s$'s pairwise commute; thus,
the $v_j$ are also $p$-nilpotent and pairwise commute.

The structure $\Theta_{k[T]}: \bG_a(k[T]) \times (M\otimes k[T]) \to (M\otimes k[T])$ of (\ref{spec})
is given by
\begin{equation}
\label{Theta}
\Theta_{k[T]}(T,m\otimes 1) \ = \ \sum_{j \geq 0} v_j(m)T^j.
\end{equation}
\end{ex}

For a $kG_{a(r)}$-module $M$, there is a natural choice of $p$-nilpotent operator on $M$
associated to a (infinitesimal) 1-parameter subgroup $\mu: \bG_{a(r)} \to \bG_a$ (see Definition \ref{infact}). 
After much experimentation, we have identified the following choice of $p$-nilpotent operator
for a rational $\bG_a$-module $M$ and a 1-parameter subgroup $\sigma_{\ul a}: \bG_a \to \bG_a$.
This definition leads to Definition \ref{Gaction} formulated for $G$ a linear algebraic group equipped
with a structure of exponential type.  We shall see in Proposition \ref{com-Ga}  how the action given 
in Definition \ref{Gaaction} is related to the action at infinitesimal 1-parameter subgroups $\bG_{a(r)} \to \bG_a$
as first considered in \cite{SFB2}.

\begin{defn}
\label{Gaaction}
Let  $M$ be a rational module for the additive group $\bG_a$ and $\sigma_{\ul a}: \bG_a \to \bG_a$
be a 1-parameter subgroup given by the (finite) sequence $\ul a = (a_0,\ldots,a_s,\dots)$.
The (nilpotent) action of $\bG_a$ on $M$ at $\sigma_{\ul a}: \bG_a \to \bG_a$
is defined to be the action of 
\begin{equation}
\label{gaact}
\sum_{s \geq 0} (\sigma_{a_s})_*(u_s) \ = \ \sum_{s \geq 0} a_s^{p^s} u_s \ \in \ k\bG_a.
\end{equation}
\end{defn}

The  equality $(\sigma_b)_*(u_s) \ = \ b^{p^s}u_s$ is confusing at first glance.  The reader can check this
as follows:  $\sigma_b: \bG_a \to \bG_a$ induces $\sigma_b^*: k[T] \to k[T], \ T \mapsto b\cdot T$.
Thus, $\sigma_b^*(T^{p^s}) = b^{p^s}\cdot T^{p^s}$, so that reading off the coefficient of  $T^{p^s}$
in the polynomial  $\sigma_b^*(f(T))$ is the same operation as reading  off  $b^{p^s}$ times 
the coefficient of  $T^{p^s}$ in the polynomial $f(T)$.

\begin{remark}
The action of $\bG_a$ on $M$ at $\sigma_{\ul a}$ is {\it not} given by the action on $\sigma_{\ul a}^*(M)$
of some naturally chosen element of $k\bG_a$.  Indeed, there does not seem to be a reasonable
choice of $\phi \in k\bG_a$ which would yield $(\sigma_{\ul a})_*(\phi)$ as a suitable alternative 
to $\sum_{s \geq 0} a_s^{p^s} u_s$.
\end{remark}

As we observed in Example \ref{ex-ga}, the action of $\bG_a$ on a finite dimensional rational
$\bG_a$-module $M$ involves only finitely many $u_s \in k\bG_a$. In other words, if $M$ is 
finite dimensional, then there exists some $r$ such that the action of $u_s$ on $M$ is trivial
for all $s \geq r$.  We next verify a similar statement for finite dimensional rational $G$ modules
whenever $G$ is equipped with a structure of exponential type.  In the terminology of 
Definition \ref{expdegree}, this proposition verifies that every finite dimensional 
rational $G$-module has bounded ``exponential degree".

\begin{prop}
\label{bound}
Let $G$ be a linear algebraic group equipped with a structure of exponential type and let
$M$ be a finite dimensional rational $G$-module.  Then there exists an integer $r$ such that 
$(\cE_B)_*(u_s)$ acts trivially on $M$ for all $s \geq r$, all $B \in \cN_p(\fg)$. 	
\end{prop}

\begin{proof}
Let $\cE: \cN_p(\fg) \times\bG_a \ \to G$ be the map giving $G$ the structure of exponential type
and consider the composition
$$\cE^* \circ \Delta_M: M \  \to \ M\otimes k[G] \ \to \ M \otimes k[\cN_p(\fg)] \otimes k[T]. $$
We choose $r$ such that the image of the composition lies in the subspace 
$M \otimes k[\cN_p(\fg)]\otimes k[T]_{<p^r}$, where $ k[T]_{<p^r}$ is
the subspace of $k[T]$ of polynomials of degree $< p^r$.  
Then for any $B \in \cN_p(\fg)(k)$ (i.e., any $k$-point
of $\cN_p(\fg)$), composition with 
evaluation at $B$ determines
$$\cE_B^* \ = \ ev_B \circ \cE^* \circ \Delta_M: M \to M \otimes k[T]_{<p^r}.$$
Since the action of $(\cE_B)_*(u_s)$ is given by composing $\cE_B^*$ with the linear map
$1\otimes u_s: M \otimes k[T] \to M$ (i.e., 
with evaluation at $T^{p^s}$), the proposition follows.
\end{proof}

\begin{ex}
\label{ex-poly}
Let $M$ be a finite dimensional {\it polynomial} $GL_n$ module of degree $d$ (see \cite{G});
thus, the comodule structure for $M$ has the form
$$\Delta_M: M \ \to \ M \otimes k[M_n]_d \ \subset \ M \otimes k[GL_n]$$
where $k[M_n]_d$ is the coalgebra of algebraic functions of degree $d$ on $M_n \simeq \bA^{n^2}$.  
The map
$exp: \bG_a \times \cN_p(\gl_n) \ \to GL_n$ extends to a map
$$exp: \bG_a \times M_n \ \to  \ M_n, \quad (s,A) \mapsto (1+ sA + \cdots + \frac{s^{p-1}}{(p-1)!}A^{p-1})$$
whose map on coordinate algebras  $exp^*: k[M_n] \to k[M_n] \otimes k[T]$ sends
$X_{i,j}$ to  a polynomial in $T$ (with coefficients in $k[M_n]$) of degree $< p$.  

Consequently, the composition 
$$exp^* \circ \Delta_M: M \ \to \ M \otimes k[M_n]_d \ \to \ M \otimes k[M_n] \otimes k[T]$$
when evaluated at any $A \in M_n$ has image contained in $M \otimes  k[T]_{\leq (p-1)d}$.
Thus, the action of $(exp_A)_*(v_j)$ on $M$ (given by the composition of $ev_A \circ exp^* \circ \Delta_M$ 
with reading off the coefficient of $T^j$) is trivial provided that $j > (p-1)d$.
This explicit bound for the vanishing of $(exp_A)_*(u_r)$ (namely, for all $r$ such that $p^r > (p-1)d$)
is stronger than the bound (\ref{constraint2}) given in a more general context at the beginning of Section \ref{sec:findim}.
The action of the product $[(exp_A)_*(v_j))]\cdot [(exp_A)_*(v_{j^\prime}))]$ 
is computed by applying $1 \otimes v_j \otimes v_{j^\prime}$
to the image of $\Delta_T \circ exp_A^* \circ \Delta_M: M \to M \otimes k[T] \otimes k[T]$.  This enables us to 
conclude that  the action of $(exp_A)_*(v_j)$ on $M$ has $(p-1)$-st power equal to 0 if $j > d$.
In particular, if $r$ satisfies $p^r > d$ (but  not necessarily satisfies $p^r > (p-1)d$), then the 
action of $(exp_A)_*(u_r)$ has $(p-1)$-st power 0.
\end{ex}

We require the following proposition to justify our definition in Definition \ref{Gaction} of the action 
of $G$ on $M$ at a 1-parameter subgroups $\cE_{\ul B}: \bG_a \to G$.

\begin{prop}
\label{commute}
Let $G$ be a linear algebraic group equipped with a structure of exponential type and let
$M$ be a finite dimensional rational $G$-module.   For any pair  $B_1, \ B_2$ of 
commuting elements of $\cN_p(\fg)$ and 
any pair $\phi_1,\phi_2: k[\bG_a] \to k$ of $p$-nilpotent elements of $k\bG_a$,
\ $(\cE_{B_1})_*(\phi_1) ,\  (\cE_{B_2})_*(\phi_2)$ \
are commuting, $p$-nilpotent elements of $kG$.
\end{prop}

\begin{proof}
Since $(\cE_{B_i})_*: k\bG_a \to kG$ is an algebra homomorphism,   $(\cE_{B_i})_*(\phi_i)$ is
$p$-nilpotent whenever $\phi$ is $p$-nilpotent.

Consider the commutative diagram
\begin{equation}
\label{multdiag}
\begin{xy}*!C\xybox{%
\xymatrix{
\bG_a \times \bG_a  \ar[rr]^{\cE_{B_1} \times \cE_{B_2}} \ar[d]_{\tau}  & & G \times G \ar[r]^\bullet \ar[d]^\tau &  G\\
\bG_a \times \bG_a  \ar[rr]^{\cE_{B_2} \times \cE_{B_1}} & & G \times G \ar[ru]^\bullet }
}\end{xy}
\end{equation}
where $\tau$ is the interchange involution of $G \times G$ and $\bullet$ is the multiplication map of $G$.
Composing the (commutative) diagram of functions induced by (\ref{multdiag}) with the (commuting) 
functionals $\phi_1, \phi_2$ yields the following commutative diagram
\begin{equation}
\label{multdiagg}
\begin{xy}*!C\xybox{%
\xymatrix{
k[G] \ar[r]^-\Delta \ar[rd]^\Delta & k[G] \otimes k[G] \ar[rr]^-{\cE_{B_2}^* \times \cE_{B_1}^*} 
\ar[d]^{\tau^*} & & k[T] \times k[T] \ar[r]^-{\phi_2\otimes \phi_1} 
\ar[d]^{\tau^*}  & k \\
& k[G] \otimes k[G] \ar[rr]^-{\cE_{B_1}^* \times \cE_{B_2}^*} & & k[T] \times k[T] \ar[ru]_{\phi_1\otimes \phi_2} \\}
}\end{xy}
\end{equation}
The asserted commutativity is the statement of the equality of upper horizontal composition of (\ref{multdiagg})
and the composition involving the lower horizontal map.
\end{proof}

Propositions \ref{bound} and \ref{commute} enable us to conclude that $\sum_{s \geq 0} (\cE_{B_s})_*(u_s) $
is a well defined $p$-nilpotent operator whenever $\ul B = (B_0,\ldots,B_s,\ldots) \in \cC_\infty(\cN_p(\fg))$.

\begin{defn}
\label{Gaction}
Let  $M$ a rational module for the linear group $G$ provided with a structure of exponential type.  
The ($p$-nilpotent) action of $G$ on $M$ at $\cE_{\ul B}: \bG_a \to G$ is defined to be the action of
\begin{equation}
\label{gact}
\sum_{s \geq 0} (\cE_{B_s})_*(u_s) \ \in \ kG
\end{equation}
on $M$ for any $\ul B \in C_\infty(\cN_p(\fg))$.
\end{defn}
\vskip .1in

\begin{ex}
\label{Xij}
Fix  some $i\not= j, \ 1 \leq i, \ j \leq N$, and let  $\psi_{i,j}: \bG_a \to GL_N$ be the root subgroup
given by the map $(\psi_{i,j})^*: k[GL_N] \to k[T]$ with $(\psi_{i,j})_*(X_{s,s}) = 1$ and
$(\psi_{i,j})_*(X_{s,t}) = \delta_{i,s}\delta_{j,t}\cdot T$ for $s \not= t$.    Thus, 
$d\psi_{i,j}: \fg_a \to \gl_N$ sends $b$ to the 
the $N \times N$ matrix whose only non-zero entry is $b$ in the $(i,j)$-position.
Then $\psi_{i,j}$ is an embedding
of exponential type and one readily checks that 
$$(\psi_{i,j})_*(\sum_{s \geq 0} (\sigma_{a_s})_*(u_s )) \ 
= \ \sum_{s \geq 0} (exp_{(d\psi_{i,j})(a_s)})_*(u_s) \ \in \ kGL_N$$
for any $\ul a \in \bA^\infty$.  

Consequently, if $M$ is a rational $GL_N$ module and $\sigma_{\ul a}: \bG_a \to \bG_a$ 
is a 1-parameter subgroup, then the action of $\bG_a$ on $\psi_{i,j}^*(M)$ at $\sigma_{\ul a}$
as defined in Definition \ref{Gaaction}
equals the action of $GL_N$  on $M$ at $\psi_{i,j}\circ \sigma_{\ul a}: \bG_a \to GL_N$ 
as defined in Definition \ref{Gaction}.
\end{ex}

We have the following functoriality extending Example \ref{Xij}.

\begin{prop}
\label{push-prop}
Let $G, \ G^\prime$ be linear algebraic groups of exponential type and let 
 $f: G  \to G^\prime$ be a map of exponential type.
Then for any $\ul B \in C_\infty(\cN_p(\fg))$, 
\begin{equation}
\label{push}
f_*(\sum_{s \geq 0}  (\cE_{B_s})_*(u_s)) \ = \ \sum_{s \geq 0} (\cE^\prime_{(df)_*(B_s)})_*(u_s)
\end{equation}

In this situation, for any rational $G^\prime$ module $M$ and any 1-parameter subgroup $\cE_{\ul B}: \bG_a \to G$, 
the action of $G$ on $f^*(M)$ at $\cE_{\ul B}: \bG_a \to G$ equals the action of $G^\prime$ on $M$ at $f \circ \cE_{\ul B}
\ = \  \cE^\prime_{(df)_*(\ul B)}: \bG_a \to G^\prime$.
\end{prop}

\begin{proof}
The equality (\ref{push}) follows from the equality 
\begin{equation}
\label{equa}
\cE^\prime_{(df)_*(\ul B)} = f \circ \cE_{\ul B}: \bG_a \to G^\prime
\end{equation}
which follows from (\ref{diffl}).
Equality \ref{equa} tells us that the action of $G^\prime$ on $M$ at $f \circ \cE_{\ul B}$ is the action of 
$\sum_{s \geq 0} (\cE^\prime_{(df)_*(B_s)})_*(u_s)$.
On the other hand, the action of any $\phi \in kG$ on $f^*M$ is the action of $f_*(\phi)$ on $M$ 
essentially by the definition of $f^*M$.  This establishes the second statement.
\end{proof}

Let $k \otimes_\phi M$ denote the base change of $M$ along the $p$-th power map $\phi: k \to k$.
The Frobenius twist $M^{(1)}$ of a rational representation $M$ of $G$ is the $k$ vector space
$k \otimes_\phi M$ with its natural $G^{(1)}$-structure; restricting along the Frobenius morphism $G \to G^{(1)}$
provides $M^{(1)}$ with the structure of a rational $G$-module.    If $G$ is defined over $\bF_p$, then 
we employ the  natural identification of $G^{(1)}$ with $G$ so that the Frobenius morphism 
becomes an endomorphism
$F: G \to G$; moreover, the action of $G^{(1)}$ on $M^{(1)}$ can be naturally identified with 
the pull-back along $F: G \to G$ of the given action of $G$ on $M$; we identify this pull-back of $M$ along
$F$ with the rational $G$-module  $M^{(1)}$.    The reader can find an exposition of such Frobenius structures in \cite{FS}.

As a companion to Proposition \ref{push-prop}, we have the following additional functoriality of our actions.
The reader should observe that the Frobenius morphism $F: G\to G$ is far from a map of exponential type; indeed,
$dF: \fg \to \fg$ is the 0-map.

\begin{prop}
\label{func-frob}
Let $i: G \hookrightarrow GL_n$ be an embedding of exponential type defined over $\bF_p$ and
let $F: G \to G$ be the Frobenius endomorphism.  Let $M$ be a rational $G$-module.
\begin{enumerate}
\item
For any $\ul B = (B_0,B_1,\ldots) \in C_\infty(\cN_p(\fg))$, 
\begin{equation}
\label{push-frob}
F_*\big(\sum_{s \geq 0}  (exp_{B_s})_*(u_s)\big) \ = \ \sum_{s \geq 1} (exp_{(B_s^{(1)})})_*(u_{s-1}).
\end{equation}
\item  
For any rational $G$-module $M$ with Frobenius twist $M^{(1)} = F^*(M)$, the action of
$G$ (identified with $G^{(1)}$) on $M^{(1)}$ at $exp_{\ul B}$ equals the action of $G$ 
on $M$ at $F\circ exp_{\ul B}$.
\end{enumerate}
Consequently, the action of $G$ on $M^{(1)}$ at $exp_{\ul B}$ is given by 
$\sum_{s \geq 1} (exp_{(B_s^{(1)})})_*(u_{s-1})$.
\end{prop}
 
 \begin{proof}
  The equality \ $F \circ (exp_B) \ = \ exp_{B^{(1)}} \circ F$ was established  
 in the proof of Proposition \ref{twist}.  This leads to the equalities
 $$F_*\big((exp_{B_s})_*(u_s)\big) \ = \ (exp_{B_s^{(1)}} \circ F)_*(u_s) \ = \ (exp_{B_s^{(1)}})_*(u_{s-1})$$
 which yields (\ref{push-frob}) thanks to (\ref{Fu}).
 
 Once we identify  the rational $G$-module $M^{(1)}$ with the pullback of $M$ 
 along $F: G \ \to \ G$, then the action of $G$ on $M^{(1)}$ at $exp_{\ul B}$ is post-composition 
 with the Frobenius morphism applied to the action of $G$ on $M$.    Thus, assertion (2) follows 
 from Proposition \ref{push-prop}.
\end{proof}

\section{Invariants for rational $\bG_a$-modules}
\label{sec:gatype}
 
In this section, we use the action of $\bG_a$ on a rational $\bG_a$-module 
at  1-parameter subgroups $\sigma_{\ul a}: \bG_a \to \bG_a$ (as introduced in Definition \ref{Gaaction})
to define invariants of rational $\bG_a$-modules.   Considering the special case $G = \bG_a$
is a good guide to Section \ref{sec:exptype}  in which we consider invariants for rational $G$-modules for
$G$  a linear algebraic group of exponential type.  The simplest and perhaps most useful
invariant is the support variety $V(\bG_a)_M$ of a finite dimensional rational $\bG_a$-module $M$ as defined
in Definition \ref{supportGa}.   Theorem \ref{Ga-items} presents some of the basic properties
of $M \mapsto V(\bG_a)_M$.  This support variety admits the refinement given in Definition \ref{def-ga-nonmax},
properties of which are presented in Theorem \ref{Ga-nonmax}.

Most of the properties of $V(\bG_a)_M$ are derived from corresponding properties for the 
support varieties for $M$ restricted to the Frobenius kernels $\bG_{a(r)} \subset \bG_a$.
Thus, a key (and confusing) comparison is made in Proposition \ref{com-Ga} between the ``linearization"
of actions at infinitesimal 1-parameter subgroups of $\bG_a$ obtained by ``twisting by $\lambda$"
the restrictions of a given 1-parameter subgroup $\sigma_{\ul a}: \bG_a \to \bG_a$ and the action at $\sigma_{\ul a}$
as defined in Definition \ref{Gaaction}.

Unlike the rest of this section which concerns the special case $G = \bG_a$, we recall the following
 definition of the (nilpotent) action 
at an infinitesimal 1-parameter subgroup $\mu: \bG_{a(r)} \to G$ for any affine group scheme
(see, for example \cite{SFB2}); in Section \ref{sec:exptype}, we shall refer back to this action in the 
case that $G$ is a linear algebraic group of exponential type. \

  If $V$ is a $k$ vector space of dimension $n$ and $\phi$ is a 
$k$-linear endomorphism of $V$, then we employ the notation 
$$JT(\phi,V) \ = \ \sum_{i=1}^p c_i[i], \quad \ \sum_i c_i \cdot i = n$$ 
for the Jordan type  of $\phi: V \to V$, indicating that the canonical Jordan form of $\phi$ consists
of $c_i$ blocks of size $i$.
 
\begin{defn}
\label{infact}
Let $G$ be an affine algebraic group scheme, $\mu: \bG_{a(r)} \to G$ be an infinitesimal 1-parameter
subgroup of $G$, and $M$ be a  $kG_{(r)}$-module (e.g, the restriction of a rational $G$-module).
Then the {\bf action of $G_{(r)}$} on $M$ at $\mu$ is defined to be the action of 
$\mu_*(u_{r-1}) \in kG_{(r)}$ on $M$, where $u_{r-1} \in k\bG_{a(r)}$ is the functional
$k[T]/T^{p^r} \to k$ which sends $f(T)$ to its coefficient of $T^{p^{r-1}}$.

The {\bf Jordan type} of a finite dimensional $G_{(r)}$-module $M$ at $\mu: \bG_{a(r)} \to G$, 
$JT_{G_{(r)},M}(\mu)$, is defined to be the Jordan type of the 
$p$-nilpotent operator $\mu_*(u_{r-1})$ on $M$,
$$JT_{G_{(r)},M}(\mu) \ \equiv \ JT(\mu_*(u_{r-1}),M).$$

The  {\bf support variety} (or, rank variety) of a $G_{(r)}$-module $M$ is defined to be the 
conical, (reduced) subvariety
$V_r(G_{(r)})_M \  \subset \ V_r(G_{(r)}) = V_r(G)$ consisting of those $ \mu: \bG_{a(r)} \to \bG_{a(r)}$
with the property that  $( \mu_*(u_{r-1}))^*(M)$ is {\it not free} as $(k[u_{r-1}]/u_{r-1}^p)$-module; 
if $M$ is finite dimensional, then $V_r(G_{(r)})_M \  \subset \ V_r(G)$ is closed \cite[6.1]{SFB2}.
\end{defn}

\vskip .1in

We now restrict to the special case that our linear algebraic group $G$ equals $\bG_a$ and refer the reader
to our initial discussion of rational $\bG_a$-modules in  Example \ref{ex-ga}.   Before we begin to introduce
invariants for rational $\bG_a$-modules using 1-parameter subgroups $\sigma_{\ul a}: \bG_a \to \bG_a$,
we describe in the following proposition how to realize rational $\bG_a$-modules through 1-parameter subgroups.

\begin{prop}  
\label{prop:constr}
Let $V_n$ be a vector space over $k$ of dimension $n$.
\begin{enumerate}
\item
Once a basis $\{ v_1,\ldots,v_n \}$ for $V$ is chosen, then a rational $\bG_a$-module structure
on $V$ is equivalent to a 1-parameter subgroup $exp_{\ul A} = \sum_{s \geq 0} exp_{A_s}\circ F^s: \bG_a \to GL_n$.
In this equivalence, the rational $\bG_a$-module is identified with $exp_{\ul A}^*(V)$,
the pull-back via $exp_{\ul A}$ of the defining representation of $GL_n$.
\item Granted this choice of basis for $V$, a $k\bG_a$-module structure on $V$ is specified
by requiring $u_s \in k\bG_a$ to act as $(exp_{\ul A})_*(u_s)$ on $V$.  Here, one views $(exp_{\ul A})_*(u_s) \in kGL_n$
as acting upon the defining $n$-dimensional vector space $V$ for $GL_n$.
\item
Given a $\bG_{a(r)}$-module structure on $V$ for some $r > 0$, there is a natural extension of this structure to
a structure of a rational $\bG_a$-module on $V$.
\end{enumerate}
\end{prop}

\begin{proof}
The first assertion is that of Proposition \ref{act}(4).  The second assertion follows from the
observation that the action of $u_s$ on $exp_{\ul A}^*(V)$ equals the action of 
$(exp_{\ul A})_*(u_s)$ on $V$, together with the identification of $k\bG_a$ given in 
Example \ref{ex-ga}.  

To extend a $\bG_{a(r)}$-structure on $V$to a rational $\bG_a$-module structure,
we use  the evident splitting of group algebras
$k\bG_a = k[u_0,\ldots,u_{r-1}]/(u_i^p) \ \to k[u_0,\ldots,u_n,\ldots] /(u_i^p) \ = \ k\bG_a$ which sends
$u_s \in k\bG_a$ to 0 if $s\geq r$ and to $u_s$ for $s < r$.

Alternatively, given a $k[G_{a(r)}]$-comodule structure $\Delta_V: V \to V\otimes k[\bG_{a(r)}]$,
we construct a $k[\bG_a]$-comodule structure on $V$ as follows.  
Denote by $\tau_r: k[\bG_{a(r)}] = k[T]/T^{p^r} \to k[T] = k[\bG_a]$ 
the map of coalgebras defined by sending $\ol{f(T)}  \in k[T]/T^{p^r}$ to the unique polynomial $f(T)$ of 
degree $< p^r$ whose reduction modulo $(T^{p^r})$ equals $\ol{f(T)}$.  Then
$$\tau_r \circ \Delta_V: V \ \to \ V \otimes k[\bG_{a(r)}] \ \to \ V \otimes k[\bG_a]$$
is an extension of $\Delta_V$ to a $\bG_a$-comodule structure on $V$.
\end{proof}

One simple corollary of Proposition \ref{prop:constr} is the following.

\begin{cor}
\label{wild}
The category of finite dimensional rational $\bG_a$-modules is wild.
\end{cor}

\begin{proof}
The group algebra $k\bG_{a(r)}$ is isomorphic as an algebra to $kE$, where $E$ is an
elementary abelian $p$-group of rank $r$.  Thus, the category of finite dimensional 
$k\bG_{a(r)}$-modules is isomorphic to the category of finite dimensional $kE$-modules.
The latter is well known to be wild if $ r \geq 2$ with $p$ odd, $r \geq 3$ for $p=2$. 
Proposition \ref{prop:constr}(3) now implies that the category of finite dimensional
$\bG_a$-modules is wild.
\end{proof}

Our goal is to investigate a given finite dimensional rational $\bG_a$-module $M$ in terms of its 
behavior along 1-parameter subgroups $\sigma_{\ul a}: \bG_a \to \bG_a$.  We first recall the
following identification of the ``linearization of the action" of the  restriction along an
infinitesimal 1-parameter subgroup of $\bG_a$.

\begin{prop} (Suslin-Friedlander-Bendel \cite[6.5]{SFB2})
\label{linapprox}
Consider the infinitesimal 1-parameter subgroup $\sigma_{\ul a}: \bG_{a(r)} \to \bG_a$
given by the $r$-tuple $(a_0,\ldots,a_{r-1}) \in \bA^r$ as in Example \ref{additive}.  Then
\begin{equation}
\label{mu}
\sigma_{\ul a *}(u_{r-1}) \ = \  
\sum_{i=0}^{r-1} a_{r-1-i}^{p^i}u_{i} + \text{higher order terms},
\end{equation}
where the higher order terms in (\ref{linapprox}) are those which are not linear in the $\{ u_i \}$'s.
\end{prop}

We point out that non-linear terms do occur in (\ref{mu}) for $r \geq 2$: 
as observed in \cite[6.5]{SFB2}, $\sigma_{\ul a}: \bG_{a(2)} \to \bG_a$ has the property that the expression
for $\sigma_{\ul a *}(u_2)$ is a sum of terms including non-zero multiples of $u_0^i u_1^{p-i}$
for each $i, \ 1\leq i < p$ associated to reading off the coefficients of $t^{p^2}$ in the expressions for 
$\sigma_{\ul a}^*(t^i)$.

 The ``reversal" of indices of the coefficients
 occurring in (\ref{mu}) when compared to the action of (\ref{gaact})
suggests the introduction of the following operation on sequences.

\begin{defn}
\label{lambda}
We consider the morphism
$$\lambda_r: \bA^\infty \to \bA^r(k), \quad (a_0,\ldots,a_{r-1},a_r,\ldots) \mapsto (a_{r-1},a_{r-2},\ldots,a_0);$$
we also let  $\lambda_r$  denote the associated map $\lambda_r: V(\bG_a) \ \to \ V_r(\bG_a)(k)$
sending $\sigma_{\ul a}: \bG_a \to \bG_a$ to $\sigma_{ \lambda(\ul a)}: \bG_a \to \bG_a$.

For each $c \geq 0$ we consider the morphisms
$$\lambda_{r+c,r}: V_{r+c}(\bG_a) \ \to \ V_r(\bG_a), \quad \sigma_{\ul a} \mapsto \sigma_{\lambda_r(\ul a)}.$$

For any $\ul a = (a_0,\ldots,a_{r+c-1}) \in \bA^{r+c}, \ c \geq 0$, we set $q_{r+c,r}(\ul a) = (a_c,,\ldots, a_{r+c-1})$
and consider the morphisms
$$q_{r+c,r}: V_{r+c}(\bG_a) \ \to \ V_r(\bG_a), \quad \sigma_{\ul a} \mapsto \sigma_{q_{r_c,r}(\ul a)}.$$
\end{defn}

\vskip .2in

\begin{remark}
\label{invlim}
For any $r> 0, \ c\geq 0$, we have the following commutative diagram
\begin{equation}
\label{limsq}
\begin{xy}*!C\xybox{%
\xymatrix{ V(\bG_a) \ar[d]_{pr_{r+c}} \ar[r]^= & V(\bG_a) \ar[d]^{pr_r} \\
V_{r+c}(\bG_a)(k) \ar[d]_-{\lambda_{r+c,r+c}} \ar[r]^-{pr_{r+c,r}} & V_r(\bG_a)(k) \ar[d]^{\lambda_{r,r}} \\
V_{r+c}(\bG_a)(k) \ar[r]_-{q_{r+c,r}} & V_r(\bG_a)(k)}
}\end{xy}
\end{equation}
the composition of whose left vertical arrows is $\Lambda_{r+c}$ and the composition of whose right
vertical arrows is $\lambda_r$.
Consequently, the space $V(\bG_a) \ \subset \  \varprojlim_r V_r(G)(k)$ which maps into the inverse limit of 
$\{ pr_{r+c,r}: V_{r+c}(\bG_a)(k) \to V_r(\bG_a)(k) \}$  through maps $pr_r: V(\bG_a) \to V_r(\bG_a)(k)$ can also 
be identified as a subspace (with the subspace topology) of the inverse limit of 
$\{ q_{r+c,r}: V_{r+c}(\bG_a)(k) \to V_r(\bG_a)(k) \}$ through maps $\lambda_r: V(\bG_a) \to  V_r(\bG_a)(k)$.

Moreover, we observe that
\begin{equation}
\label{pr=}
pr_r = \lambda_{r+c,r} \circ \lambda_{r+c}: V(\bG_a) \to V_r(\bG_a)(k)
\end{equation}
and that $\lambda_{r,r}$ is an involution.
\end{remark}

\begin{defn}
\label{finitedegree}
A rational $\bG_a$-module $M$ is said to have degree $< p^r$ if the coaction
$\Delta_M: M \to M\otimes k[T]$ satisfies the condition that $\Delta_M(M) \subset M\otimes k[T]_{< p^r}$.

In particular, any finite dimensional rational $\bG_a$-module $M$ has degree $< p^r$ 
for sufficiently large $r$.
\end{defn}

Since (generalized) support varieties are determined by the linearizations of $p$-nilpotent operators (in
the sense of \cite[1.13]{FPS}),
the following proposition demonstrates a promising connection between actions at 1-parameter subgroups
and infinitesimal 1-parameter subgroups.

\begin{prop}
\label{com-Ga}
Let $M$ be a rational $\bG_a$-module of degree $< p^r$.
 Then for any $\ul a  \in \bA^\infty$,
the action of $\bG_a$ on $M$ at $\sigma_{\ul a}: \bG_a \to \bG_a$ (see (\ref{gaact}) ) equals the 
{\it linearization} of the action of $\bG_{a(r)}$ on $M$ at 
$ \sigma_{\lambda_r(\ul a)}: \bG_{a(r)} \to \bG_a$.

In other words, for $\ul b = (b_0,\ldots,b_{r-1}) \ = \ \lambda_r(\ul a)$,  the action of
$\sum_{s \geq 0} a_s^{p^s} u_s$ on $M$ equals the action of  $\sum_{i = 0}^{r-1} b_{r-1-i}^{p^i}u_i$,
which is the sum of the linear terms of  $\sigma_{\ul b^*}(u_{r-1})$ as in (\ref{mu}).
\end{prop}

\begin{proof}
Recall that $u_r \in k[T]^*$ sends a polynomial $p(T)$ to the coefficient of $T^{p^r}$ of $p(T)$.
By (\ref{pairing}),  the action of 
$u_s$ on $M$ is given by sending $m\in M$ to $(id_M\otimes u_s)(\Delta_M(m))$.
Thus, for $r$ chosen sufficiently large as in the statement
of the proposition, the action of $u_r$ on $M$
is trivial.  

Consequently, the action of  $\sum_{s \geq 0} a_s^{p^s} u_s$ on $M$ equals the action
of $\sum_{s=0}^{r-1} a_s^{p^s} u_s$ on $M$.   Unravelling the definition of $\lambda_r(-)$, we easily see
that $\sum_{s=0}^{r-1} a_s^{p^s} u_s \ = \ \sum_{i = 0}^{r-1} b_{r-1-i}^{p^i}u_i$, thereby proving the 
proposition.
\end{proof}

The following definition of support varieties for finite dimensional rational $\bG_a$-modules will serve
as our model in the next section for the definition of support varieties for more general linear algebraic
groups.  
 
\begin{defn}
\label{supportGa}
Let $M$ be a  rational $\bG_a$-module.  We define the {\bf support variety} of $M$ as the 
subset \ $V(\bG_a)_M \ \subset \ V(\bG_a)$ \ with the subspace topology
consisting of those 1-parameter subgroups $\sigma_{\ul a}: \bG_a \to \bG_a$ at which
the action (in the sense of Definition \ref{Gaction}) is not free; in other words, the action of $M$
of $\sum_{s \geq 0} (\sigma_{a_s})_*(u_s) \in k\bG_a$ is not free.

 If $M$ is finite dimensional, then $V(\bG_a)_M \ \subset \ V(\bG_a)$ consists of those
1-parameter subgroups $\sigma_{\ul a}$ at which the Jordan type of $M$ at $\sigma_{\ul a}$
(in the sense of Definition \ref{GaJT} below) has some block of size less than $p$.
\end{defn}

We provide a list of good properties for $M \  \mapsto \   V(\bG_a)_M$, using the analogues of
these properties established for infinitesimal group schemes (see \cite{SFB2}).   For the first 
property, we require that $M$ has degree $< p^r$ for some $r > 0$, for the second we require
the stronger condition that $M$ be finite dimensional (also required in the case of rational
modules for infinitesimal group schemes), and for the remaining properties we place no 
condition on $M$.

\begin{thm}
\label{Ga-items}
Let $M$ be a rational $\bG_a$-module.
\begin{enumerate}
\item  If $M$ has degree $< p^r$, then $V(\bG_a)_M \ = \ \lambda_r^{-1}(V_r(\bG_a)_M(k))$ 
(which equals $pr_r^{-1}(\lambda_{r,r}(V_r(\bG_a)_M(k)))$).
\item
If $M$ is finite dimensional, then $V(\bG_a)_M \subset V(\bG_a)$ is closed.
\item
$V(\bG_a)_{M\oplus N} = V(\bG_a)_M \cup V(\bG_a)_N.$
\item
$V(\bG_a)_{M\otimes N} = V(\bG_a)_M \cap V(\bG_a)_N.$
\item
 If $0 \to M_1 \to M_2 \to M_3 \to 0$ is a short exact sequence
of rational $\bG_a$-modules, then the support variety $V(\bG_a)_{M_i}$ of one of the
$M_i$'s is contained in the union of the support varieties of the other two.
\item 
$V(\bG_a)_{M^{(1)}} \ = \ \{ \sigma_{(a_0,a_1,\ldots)} \in V(\bG_a) :
 \sigma_{(a_1^p, a_2^p,\ldots)} \in V(\bG_a)_M\} .$
\item  For any $r >0$, the restriction of $M$ to $k\bG_{a(r)}$ is injective (equivalently, projective) 
if and only if the intersection of $V(\bG_a)_M$ with the subset 
$\{ \sigma_{\ul a}: a_s = 0, s \geq  r \} \ \subset \ V(\bG_a)$ equals $\{ \sigma_{\ul 0} \}$.
\end{enumerate}
\end{thm}

\begin{proof} Comparing Definitions \ref{infact} and \ref{supportGa}, we see that to prove Property (1) it suffices to 
compare the actions of $\sigma_{\lambda_r(\ul a),*}(u_{r-1})$ and of $\sum_{s \geq 0} (\sigma_{a_s})_*(u_s)$.
Thus, Property (1) for $M$ finite dimensional is a consequence of Proposition \ref{com-Ga}
and \cite[2.7]{FPS}, which compares maximal Jordan types for finite dimensional modules pulled-back 
via a flat map and its linearization.  For $M$ infinite dimensional, we use the finite dimensional case just proved
together with  \cite[4.6]{FP2} which asserts that comparing projectivity pulled back along two
$\pi$-points for all finite dimensional representations suffices to compare projectivity of pull-backs for 
arbitrary representations.   Property (2) follows from 
property (1), the fact that $V_r(\bG_a)_M$ is closed in $V_r(G)$ whenever $M$ is finite
dimensional,  and the defining property of the topology on
$V(\bG_a)$.

Properties (3), (4), and (5) are readily checked by checking at one $\sigma_{\ul a} \in V(\bG_a)$ at a time.
Namely, for any $\sigma_{\ul a} \in V(\bG_a)$, we restrict the action of $k\bG_a$ on $M$ to $k[u]/u^p$, where
$k[u]/u^p \to k\bG_a$ is given by sending $u$ to $\sum_{s \geq 0} a^{p^s}u_s$.  Thus, the verification
at some $\sigma_{\ul a} \in V(\bG_a)$ reduces to the verification of these properties for $k[u]/u^p$-modules,
which is essentially trivial.

Property (6) follows from Proposition \ref{func-frob}
which tells us that the action of $\sigma_{(a_0,a_1,\ldots)}$ on $M^{(1)}$ equals the action of
$ \sigma_{(a_1^p,a_2^p,\ldots)} $ on $M$.  

For any $\sigma_{\ul a}$ with $a_s = 0$ for $s \geq r $, Proposition \ref{com-Ga} tells us that the 
action of $\bG_a$ on $M$  at $\sigma_{\ul a}$ equals the linearization of the action
of $(\sigma_{\lambda_r(\ul a)})_*(u_{r-1})$;   
as shown in \cite[4.6]{FP2} (for possibly infinite dimensional modules), one of these actions is free if and only if the other is.
Since $\lambda_r$ is an involution on the set of involutions $\ul a$ with  $a_s = 0$ for $s \geq r $, the condition that the action of $(\sigma_{\lambda_r(\ul a)})_*(u_{r-1})$
on $M$ is free for all $\ul a \not= 0$ with $a_s = 0$ for $s \geq r$ is equivalent to the condition that $M$ is injective
as a $\bG_{a(r)}$-module.  As defined, for example in \cite[5.1]{FP2}, the 
 subset  of those $ \sigma_{\ul a} \in  \ V(\bG_a)$ at which $M$ is not free is
by definition $V(\bG_a)_M$.    Property (7) now follows.
\end{proof}

As recalled in the proof of Corollary \ref{wild}, 
the category of $\bG_{(a(r)}$-modules is equivalent to the category of $kE$-modules where
$E$ is the elementary abelian $p$-group $\Z/p^{\times r}$.
Since many examples of support varieties (equivalently, of rank varieties) for elementary
abelian $p$-groups have been computed, Proposition \ref{prop:constr}(3) together with Theorem 
\ref{Ga-items}(1) provides many
explicit examples.  The following corollary is a simple consequence of one aspect of our
knowledge of support varieties for elementary abelian $p$-groups.

\begin{cor}
For any  closed and conical subset  $X \subset V_r(\bG_{a(r)}) $
(i.e., zero locus of homogeneous polynomial equations), there exists some finite dimensional
rational $\bG_a$-module $M_X$ such that 
$$V(\bG_a)_{M_X} \ = \ pr_r^{-1}(X) \  \subset \ V(\bG_a).$$
\end{cor}

\begin{proof}
As in Remark \ref{confuse},  $pr = \lambda_{r,r} \circ \lambda_r:  V(\bG_a) \to V_r(\bG_a)(k)$.
The construction of Carlson's $L_\zeta$-modules
determines a finite dimensional  $\bG_{a(r)}$-module $M$ such that 
$V_r(\bG_{a(r)})_M) \ = \  \lambda_{r,r}(X)$ (see \cite[7.5]{SFB2}).  Thus, the corollary follows from 
Proposition \ref{prop:constr}(3) and Theorem \ref{Ga-items}(1).
\end{proof}

In the definition below, we define the ``Jordan type" of a finite dimensional, rational 
$\bG_a$-module $M$ at a
1-parameter subgroup $\sigma_{\ul a}: \bG_a \to \bG_a$ as the Jordan type of the 
$p$-nilpotent operator associated to the action of $\bG_a$ on $M$ at $\sigma_{\ul a}$.  

\begin{defn}
\label{GaJT}
Let $M$ be a finite dimensional rational $\bG_a$-module, $\ul a \in \bA^\infty$, 
and $\sigma_{\ul a}: \bG_a \to \bG_a$.   We define the {\bf Jordan type} of $M$ at $\sigma_{\ul a}$
by setting
$$JT_{\bG_a,M}(\sigma_{\ul a}) \quad \equiv \quad JT(\sum_{s \geq 0} a_s^{p^s}u_s,M),$$
 the Jordan type of the action of $\bG_a$ on $M$ at $\sigma_{\ul a}$ (see Definition \ref{Gaaction}).
\end{defn}

\vskip .1in

We remind the reader of the partial ordering on Jordan types of an endomorphism of an
$m$-dimensional vector space: 
\begin{equation}
\label{eq:partial}
 c_1[1] + \cdots +c_p[p] \ \leq \  b_1[1]+ \cdots + b_p[p]  \ \Leftrightarrow \  \sum_{i=j}^p c_i\cdot i \leq
 \sum_{i=j}^p i \cdot b_i, \ \forall j
 \end{equation}
for $m = \sum c_i \cdot i = \sum b_i \cdot i.$
We also remind the reader that for any finite dimensional $\bG_{a(r)}$-module $M$
there is an infinitesimal 1-parameter subgroup $\mu: \bG_{a(r)} \to \bG_{a(r)}$
at which the Jordan type of $M$ is ``strictly maximal", greater or equal to the 
Jordan type of $M$ at any infinitesimal 1-parameter subgroup of $\bG_{a(r)}$.  This is 
 a reflection of the fact that $V_r(\bG_a)$ is irreducible.  

In the following theorem, we verify for $r$ 
sufficiently large for a given finite dimensional rational $\bG_a$-module $M$ that 
the maximum Jordan type of $M$ as a rational $\bG_a$-module equals the  
Jordan type of $M$ as a $\bG_{a(r)}$-module.
The reader will observe an unavoidable confusion of notation:
if $\ul a =  (a_0,\ldots,a_{r-1},0,0,\ldots)$, then $JT_{\bG_a,\sigma_{\ul a}}(M)$ equals
$JT_{\bG_{a(r)},\sigma_{\lambda_r(\ul a)}}(M)$ and not $JT_{\bG_{a(r)},\sigma_{pr_r(\ul a)}}(M)$.
The  compatibility of Jordan types of $M$ restricted to $\bG_{a(r)}$ for $r >> 0$ is achieved 
thanks to our twisting functions $\lambda_r$.

\begin{thm}
\label{maximal}
Let $M$ be a  finite dimensional rational $\bG_a$-module of degree $< p^r$.   
Let $\ul b = (b_0,\ldots,b_{r-1})$ be chosen so that
$JT_{\bG_{a(r)},M}(\sigma_{\ul b})$ is (strictly) maximal
among partitions of $dim(M)$  occurring as Jordan types of $M$ at (infinitesimal)
1-parameter subgroups of $\bG_{a(r)}$.
\begin{enumerate}
\item 
For any 1-parameter subgroup $\sigma_{\ul a}: \bG_a \to \bG_a$
with $\lambda_r(\ul a) = \ul b$, 
\begin{equation}
\label{confuse}
JT_{\bG_a,M}(\sigma_{\ul a}) \ = \  JT_{\bG_{a(r)},M}(\sigma_{\ul b}).
\end{equation}
In other words, the Jordan type of $M$  at $\sigma_{\ul a}$ as a $\bG_a$-module 
equals the Jordan type of $M$ at $\sigma(\ul b)$ as a $\bG_{a(r)}$-module.
\item
$JT_{\bG_{(a(r)},M}(\sigma_{\ul b})$  equals the maximum among the
Jordan types $JT_{\bG_a,M}(\sigma_{\ul a}),\ \ul a \in\bA^\infty$.
\item
For any $c \geq 0$, 
$V_{r+c}(\bG_a)_M  \ = \  q_{r+c,r}^{-1}(V_r(\bG_a)_M).$
\end{enumerate}
\end{thm}

\begin{proof}
The first statement follows from Proposition \ref{com-Ga} and  \cite[1.13]{FPS} (see also \cite[2.7]{FPS}).
Namely, the  fundamental result concerning  the maximal Jordan type of a $k[u_1,\ldots,u_{r-1}]/(u_i^p)$-module
$M$ asserts that if this  Jordan type is achieved as the pull-back via some flat map
$\alpha: k[t]/t^p \to k[u_1,\ldots,u_{r-1}]/(u_i^p)$, then it is achieved by the (necessarily, flat) map obtained by 
sending $t$ to the linear part of  $\alpha(t)$ (i.e., linear in the $u_i$'s).

The second statement follows from the observation that the choice of $r$ guarantees that the action of $\bG_a$
on $M$ at $\sigma_{\ul a}$ equals the action of $\bG_a$ on $M$ at $\sigma_{pr_r(\ul a)}$ which in turn
equals the linearization of the action of $\bG_{a(r)}$ on $M$ at $\sigma_{\lambda_r(\ul a)}$
as seen in Proposition \ref{com-Ga}, together with the fact that 
 the maximum Jordan type of these linearizations is the maximal
Jordan type of $M$ as a $\bG_{a(r)}$-module.

To prove the last assertion, first observe that $\sigma_{\ul b} \notin V_{r+c}(\bG_a)_M$ if and only if the Jordan
type has all blocks of size $p$.   Of course, if the Jordan type has all blocks of size $p$, then
it is necessarily maximal.
For $r$ as in the statement of the proposition (and $M$ of degree $< p^r$)
and any $\sigma_{\ul a} \in V(\bG_a)$, we have
$$JT_{\bG_a,M}(\sigma_{\ul a}) \ = \ JT_{\bG_{a,(r+c)},M}(\lambda_{r+c}(\sigma_{\ul a})) \ = 
\ JT_{\bG_{a(r)},M}(\lambda_r(\sigma_{\ul a}))$$
if any one of those three Jordan types is maximal among the Jordan types 
$JT_{\bG_a,M}(\sigma_{\ul a})$, $\ul a \in\bA^\infty$.
Thus, the last assertion follows from the observation that $q_{r+i,r} \circ \lambda_{r+c} = \lambda_r$.
\end{proof}

In \cite{FP2}, the author and J. Pevtsova consider invariants for rational modules for finite
group schemes which are finer than support varieties.  We develop the extension of these
``generalized support varieties" for rational $\bG_a$-modules.

The following proposition, essentially found in \cite[2.8]{FP2},
is a generalization of the topological property of Theorem \ref{Ga-items}(2).

\begin{prop}
\label{pord}
Let $M$ be a finite dimensional rational $\bG_a$-module of dimension $n$
and let $[c] = \sum_{i=0}^p c_i [i]$ be a partition of $n$.  Then
$$V^{\leq [\ul c]}(\bG_a)_M \ \equiv \ \{ \sigma_{\ul a}: JT_M(\sigma_{\ul a}) \leq [\ul c] \}$$
is a closed subspace of $V(\bG_a)$.

Equivalently, for each $j, 0 < j < p$, the subset of  those $\sigma_{\ul a} \in V(\bG_a)$ 
such that the rank of 
$(\sum_{i=0}^\infty a_i^{p^i}u_i)^j: M \to M$ is less than some fixed integer $c$
(which is equivalent to the condition that the rank of $(\sum_{i=0}^\infty a_i^{p^i}u_i)^j: M \to M$
is $\leq c-1$)  is closed. 
\end{prop}

\begin{proof}
The equivalence of the two statements follows from (\ref{eq:partial}).
 Namely, for a $p$-nilpotent operator $u$ on a vector space $M$ of dimension $m$, the Jordan type
 of $u$ equals $c_1[1] + \cdots +c_p[p]$ if and only if the rank of $u^j$ equals $ \sum_{i=j}^p c_i\cdot (i-j)$.
 
 Using Theorem \ref{Ga-items}(1), it suffices to replace $V(\bG_a)$ by $V_r(\bG_a) = V(\bG_{a(r)})$
 with $r$ chosen so that $\Delta_M(M) \subset M \otimes k[T]_{<p^r}$.  We consider the
 $k[x_0,\ldots,x_{r-1}]$-linear  operator 
 $$\Theta_M: M \otimes k[x_0,\ldots,x_{r-1}] \  \to \ M \otimes k[x_0,\ldots,x_{r-1}], \quad
 m \mapsto \sum_{s=0}^{r-1} u_s( m) \otimes x_s^{p^s}.$$
Specializing $\Theta_M$ at $(a_0,\ldots, a_{r-1})$, 
 we obtain $\Theta_M \mapsto \sum_{s=0}^{r-1} a_s^{p^s} u_s$. 
Applying Nakayama's Lemma as in \cite[4.11]{FP3} to $Ker\{ \Theta_M^j\},\ 1 \leq j < p$, 
we conclude that the rank of the $j$-th-power of the specialization of $\Theta_M$ is
lower semi-continuous on $\bA^r$; thus the subset of those $\sigma_{\ul a} \in V_r(\bG_a)$ 
such that the rank of 
$(\sum_{s=0}^{r-1} a_s u_s^{p^s})^j: M \to M$ is less than some fixed integer $c$ is closed;
of course, this is equivalent to this rank being $\leq c-1$.
\end{proof}

Using Proposition \ref{pord}, we introduce the (affine) non-maximal $j$-rank variety
$V^j(\bG_a)_M$ for a finite dimensional rational $\bG_a$-module $M$ and an integer $j, 
\ 1 \leq j < p$.   For $G$ an infinitesimal group scheme, $V^j(G)_M$ was defined in \cite[4.8]{FP2}.   

\begin{defn} 
\label{def-ga-nonmax}
For any finite dimensional rational $\bG_a$-module $M$ and any $j, \ 1 \leq j < p$, 
we define the  the (affine) non-maximal $j$-rank variety of $M$
$$V^j(\bG_a)_M \subset \ V(\bG_a)$$
to be the  subset of those $\sigma_{\ul a}: \bG_a \to \bG_a$ such that 
either $\ul a = 0$ or the rank of $(\sum_{s\geq 0} a_s^{p^s} u_s)^j: M \to M$ is not maximal. 

As in \cite[4.8]{FP2}, for any $r >0$  and any $j, \ 1 \leq j < p$, we similarly define
$$V^j(\bG_{(a(r)})_M \subset \ V(\bG_{a(r)})$$
to be the  subset of those $\mu_{\ul a}: \bG_{a(r)} \to \bG_{a(r)}$ such that 
either $\ul a = 0$ or the rank of $\mu_{\ul a*}(u_{r-1}^j): M \to M$ is not maximal.

\end{defn}

Various examples  of $V^j(\bG_a)_M$ are given in \cite{FP2}, thanks to  Proposition \ref{prop:constr}(3)
and the first statement of the following theorem.

\begin{thm}
\label{Ga-nonmax}
The non-maximal $j$-rank variety $V^j(\bG_a)_M$ of a finite dimensional rational $\bG_a$-module
satisfies the following properties:
\begin{enumerate}
\item  If $M$ has degree $< p^r$, then $V^j(\bG_a)_M \ = \ pr_r^{-1}(\lambda_{r,r}(V^j_r(\bG_a)_M(k)))$.
\item
$V^j(\bG_a)_M  \subset V(\bG_a)$ is a proper closed subspace.
\item
$V^j(\bG_a)_M$ is a subspace of $V(\bG_a)_M$, with equality if and only if the action of
$\bG_a$ at some $\sigma_{\ul a}$ has all Jordan blocks of size $p$.
\item  For any $r > 0$, the restriction of $M$ to $k\bG_{a(r)}$ is a module of constant 
$j$-rank if and only if the intersection of $V^j(\bG_a)_M$ with the subset 
$\{ \sigma_{\ul a}: \ul a \not= 0, \ a_s = 0, s \geq  r \} \ \subset \ V(\bG_a)$ is empty.
\end{enumerate}
\end{thm}

\begin{proof}
As observed in \cite[3.5]{FP2}, Theorem \ref{maximal} remains valid if  maximal Jordan
type is replaced by  maximal $j$-rank.  Modifying the proof of Theorem \ref{Ga-items}(1) 
by replacing maximal Jordan type by maximal $j$-rank gives a proof of Property (1).

The containment $V^j(\bG_a)_M \ \subset \ V(\bG_a)_M$ is immediate  from the
definition of $V^j(\bG_a)_M$, as is the assertion that $V^j(\bG_a)_M$ is necessarily
a proper  subset of $V(\bG_a)$.  The assertion that $V^j(\bG_a)_M$
is closed in $V(\bG_a)$ is a restatement of the second assertion of Proposition \ref{pord}.
(Property (2) also follows from Property (1) together with the fact that $V^j(\bG_{(a(r)})_M 
\subset V(\bG_{(a(r)})$ is closed.)
 
Equality of $V^j(\bG_a)_M $ and $V(\bG_a)_M$ occurs if and only if the  maximal $j$-rank 
occurs exactly at those $\sigma_{\ul a}$ at which the action of $\bG_a$ has all blocks of 
size $p$.   If there is some such $\sigma_{\ul a}$ at which the action of $\bG_a$ has all blocks of 
size $p$, then the $j$-rank at some $\sigma_{\ul b}$ equals that at $\sigma_{\ul a}$ if and only
if $\sigma_{\ul b}$ also has all blocks of size $p$ (the $j$-rank of a matrix with  a single Jordan
block of size $\ell \leq  p$ is $min\{ 0, \ell-j\}$).  On the other hand if there does not exist 
some $\sigma_{\ul a}$ at which the action of $\bG_a$ has all blocks of size $p$, then $V(\bG_a)_M
= V(\bG_a)$, whereas it is tautological that the $j$-rank is non-maximal on a proper subset
of $V(\bG_a)$.

Finally, to prove property (4), we recall that a $\bG_{a(r)}$-module $M$ is a module of constant $j$-rank
if and only if $V^j(\bG_{a(r)})_M - \{ 0 \}$ is empty (\cite{FP2}).  Thus, this property is 
proved by a slight adaption of the proof of Theorem \ref{Ga-items}(7).
\end{proof}


\section{Support varieties for rational $G$-modules, $G$ a linear algebraic group of exponential type }
\label{sec:exptype}

In this section, we extend the results of \S 3 from the special case $G = \bG_a$ to linear
algebraic groups $G$ equipped with a structure of exponential type.  All simple algebraic
groups of  classical type are groups of exponential type; as remarked in Example \ref{Sobaj}, 
other examples are reductive algebraic groups, their parabolic subgroups, and
the unipotent radicals of parabolic subgroups subject to a condition on $p$
depending upon the type of $G$.

The formalism given for $G = \bG_a$ applies to this more 
general context with very little change, and we do not repeat those arguments
of \S 3 which apply essentially verbatim.  What enables this extension of Section \ref{sec:gatype} is
Proposition \ref{sobaje}, an interpretation of a  result of P. Sobaje.
In particular, Theorems \ref{G-items} and \ref{G-nonmax} extend to rational $G$-modules
the basic Theorems \ref{Ga-items} and \ref{Ga-nonmax} for rational $\bG_a$-modules.

We begin with the evident extension of Definition \ref{lambda}.  We remind the
reader that every 1-parameter subgroup $\bG_a \to G$ is of the form 
$\cE_{\ul B}$ if $G$ is provided with a structure of exponential type.

\begin{defn}
\label{Lambda}
Let $G$ be a linear algebraic group equipped with a structure of exponential type.  For
any $\ul B \in \cC_\infty(\cN_r(\fg))$ and any $r > 0$,  we set 
$\Lambda_r(\ul B) \ = \ (B_{r-1},B_{r-2},\ldots,B_0)$
and similarly define $\Lambda_{r+c,r}.$
We define 
$$\Lambda_r: V(G) \ \to \ V_r(G)(k), \quad \cE_{\ul B} \mapsto \cE_{\Lambda_r(\ul B)} \circ i_r,$$
$$\Lambda_{r+c,r}: V_{r+c}(G) \ \to \ V_r(G), \quad \cE_{\ul B} \circ i_{r+c} \mapsto \cE_{\Lambda_r(\ul B)} \circ i_r.$$
\end{defn}

The key observation which enables the formalism of  Section \ref{sec:gatype} to be extended 
to rational $G$-modules for $G$ equipped with a structure of exponential type is the following proposition, 
essentially an interpretation in our context of Proposition 2.3 of \cite{S1}.  
This can be viewed as a generalization of Proposition 
\ref{linapprox}.   We give an overview of Sobaje's proof which we shall use in the proof of our
refinement in Proposition \ref{sobaje-max}.

\begin{prop} (cf. Sobaje, \cite[2.3]{S1}) 
\label{sobaje}
Let $G$ be a linear algebraic group equipped with a structure of  exponential type, $\fg = Lie(G)$, 
and $r$ a positive integer.  Then for any 1-parameter subgroup $\cE_{\ul B}: \bG_a \to G$ 
and  any finite dimensional $kG_{a(r)}$-module $M$, the 
pullback of $M$ along the ($k$-rational) $\pi$-point 
\begin{equation}
\label{p1}
k[u]/u^p \to kG_{(r)}, \quad u \mapsto (\cE_{\ul \Lambda_r(B)})_*(u_{r-1}) 
\end{equation}
is projective if and only if the pullback of $M$ along the map of $k$-algebras
\begin{equation}
\label{p2}
k[u]/u^p \to kG_{(r)}, \quad u \mapsto \sum_{s=0}^{r-1} (\cE_{B_s})_*(u_s)
\end{equation}
is projective.
\end{prop}

\begin{proof}
By Proposition \ref{commute}, $\sum_{s=0}^{r-1} (\cE_{B_s})_*(u_s)$ is a sum of $p$-nilpotent,
pairwise commuting elements of $kG$ and thus (\ref{p2}) is well defined.   

The proof of  \cite[2.3]{S1} applies essentially verbatim.    Sobaje's 
proof proceeds by factoring $\cE_{\ul B}$ as 
$$ \cE_{\ul B} \ = \ \Phi_r \circ \Psi_r: \bG_a \to \bG_a^{\times r} \to \bG_a, 
\quad \Phi_r = \cE_{B_0} \bullet \cdots \bullet  \cE_{B_{r-1}}, \ 
\Psi_r = \times_{s=0}^{r-1} F^s.$$
Sobaje observes that a simple tensor $x_0 \otimes \cdots \otimes x_{r-1} \in k(\bG_a^{\times r})$
is sent by $\Phi_{r*}$ to the product $\cE_{B_0*}(x_0)\bullet \cdots \bullet \cE_{B_{r-1}*}(x_{r-1})$.  
He then verifies that $\Psi_{r_*}: k\bG_a \to
k(\bG_a^{\times r})$ sends $u_{r-1}$ to the sum of simple tensors of the form 
$u_{s,r-s} \equiv 1\otimes \cdots 1 \otimes u_s \otimes 1 \cdots \otimes 1$ 
(with $u_s$ in the $r-1-s$ position) plus a sum
of terms involving a product of two or more $p$-nilpotent terms in $k(\bG_a^{\times r})$.

Thus, $\Phi_{r_*}$ sends $\Psi_{r*}(u_{r-1})$ to the sum $\sum_{s=0}^{r-1} \cE_{B_{r-s-1}*}(u_s)$
plus  the image under  $\Phi_{r_*}$ of a sum of terms involving a product of two or more
 $p$-nilpotent elements in $k(\bG_a^{\times r})$.
Consequently, by \cite[6.4]{SFB2} applied to the abelian subalgebra of $kG$ generated by the
image of $\Phi_{r_*}$, the restriction of $M$ along (\ref{p1}) is  projective  if 
and only if the restriction of $M$ along (\ref{p2}) is projective for any finite dimensional 
$kG_{(r)}$-module $M$.  
\end{proof}

Recall that a $\pi$-point of an affine group scheme $G$ is a left flat map $\alpha: K[t]/t^p \to KG$
which factors through the group algebra of some commutative subgroup scheme $C_K \subset G_K$.
In Proposition \ref{sobaje}, the map (\ref{p1}) is a $\pi$-point, factor through the group algebra
of the image of the 1-parameter subgroup $\cE_{\ul B})$.
Using this proposition, we can conclude that (\ref{p2}) is also a $\pi$-point, thereby allowing
us to compare  Jordan types using the results of \cite{FPS}.

\begin{prop}
\label{sobaje-max}
Let $G$ be a linear algebraic group equipped with a structure of  exponential type, $\fg = Lie(G)$, 
and $r$ a positive integer.    Then 
(\ref{p2}) is a $\pi$-point of $G_{(r)}$ equivalent to the $\pi$-point (\ref{p1}).

Let $M$ a finite dimensional rational $G$-module.
Let $[\ul c] = \sum_{i=1}^p[c_i]$ be a  Jordan type maximal
among the Jordan types of $M$ at infinitesimal 1-parameter subgroups $\mu: \bG_{a(r)} \to G$.
Then the pullback of $M$ along the  $\pi$-point (\ref{p1})
has Jordan type $[\ul c]$ if and only if the pullback of $M$ along the $\pi$-point (\ref{p2})
has Jordan type $[\ul c]$.
\end{prop}

\begin{proof}
We first consider the special case $M = kG_{(r)}$.
Then $M$ is free as a (left) $kG_{(r)}$-module and since (\ref{p1}) is flat, the restriction of $M$ 
along (\ref{p1}) is a free $k[u]/u^p$-module.  By Proposition \ref{sobaje}, the restriction of $M$ 
along (\ref{p2}) is thus also a free $k[u]/u^p$-module.   Consequently, (\ref{p2}) is flat.
As observed in the proof of Proposition \ref{sobaje}, $\sum_{s=0}^{r-1} (\cE_{B_s})_*(u_s)$
lies in $\phi_{r*}(k\bG_{a(r)}^{\times r})$.  The latter is the group algebra $kC$ of the
abelian subgroup scheme of $G_{(r)}$ defined as the image under $\phi_r$ of $\bG_{a(r)}^{\times r}$.
Thus, (\ref{p2}) is a $\pi$-point.

Now, let $M$ denote an arbitrary finite dimensional rational $G$-module.
Proposition 4.2 is exactly the statement that the $\pi$-points (\ref{p1}) and (\ref{p2}) are equivalent.
The second statement of the proposition follows from the first and the independence of 
the Jordan type of $\pi$-points which are equivalent and for which the Jordan type of
$M$ is  maximal for one of the $\pi$-points \cite[3.5]{FPS}.
\end{proof}

Propositions \ref{sobaje} and \ref{sobaje-max} suggest the following extension of Definition \ref{supportGa}.  

\begin{defn}
\label{supportG}
Let $G$ be a linear algebraic group equipped with a structure of  exponential type and let
$M$ be a  rational $G$-module. 
We define the {\bf support variety} of $M$ to be the subset $V(G)_M \ \subset \ V(G)$ consisting 
of those $\cE_{\ul B}$ such that $M$ restricted to $k[u]/u^p$ is not free, where 
$u = \sum_{s \geq 0} (\cE_{B_s})_*(u_s) \in kG$ (as in (\ref{gact})).

For $M$ finite dimensional, we define the {\bf Jordan type} of $M$ as a rational $G$-module at $\cE_{\ul B}$ by
$$JT_{G,M}(\cE_{\ul B})\quad \equiv \quad JT(\sum_{s \geq 0} (\cE_{B_s})_*(u_s),M),$$
the Jordan type of the action of $G$ on $M$ at $\cE_{\ul B}$ (see Definition \ref{Gaction}).
For such a  finite dimensional rational $G$-module $M$, 
$V(G)_M \ \subset \ V(G)$ consists of those 1-parameter subgroups $\cE_{\ul B}$ 
such that some block of the Jordan type of $M$ at $\cE_{\ul B}$ has size $< p$.
\end{defn}

We proceed to verify that this definition of Jordan type satisfies the ``same" list of properties
as that of Theorem \ref{Ga-items}.  First, we require the following definition,  closely
related to the formulation of $p$-nilpotent degree given in \cite[2.6]{F1}.

\begin{defn} (cf. \cite[2.5]{F1})
\label{expdegree}
Let $G$ be a linear algebraic group equipped with a structure of exponential type and let $M$ be a 
rational $G$-module.  Then $M$ is said to have  exponential degree $< p^r$
if $(\cE_B)_*(u_s)$ acts trivially on $M$ for all $s \geq r$, all $B \in \cN_p(\fg)$.
\end{defn}

For example, Proposition \ref{bound} tells us that every finite dimensional $G$-module $M$ has
exponential degree $< p^r$ for $r$ sufficiently large.

\begin{thm}
\label{G-items}
Let $G$ be a linear algebraic group equipped with a structure of  exponential type and
$M$ a rational $G$-module
\begin{enumerate}
\item  
If $M$ has exponential degree $< p^r$, then 
$V(G)_M \ = \ \Lambda_r^{-1}(V_r(G)_M(k)))$ (which equals $pr_r^{-1}(\lambda_{r,r}(V_r(G)_M)(k))$).
\item
If $M$ is finite dimensional, then $V(G)_M \subset V(G)$ is closed.
\item
$V(G)_{M\oplus N} = V(G)_M \cup V(G)_N.$
\item
$V(G)_{M\otimes N} = V(G)_M \cap V(G)_N.$
\item
If $0 \to M_1 \to M_2 \to M_3 \to 0$ is a short exact sequence
of rational $G$-modules, then the support variety $V(G)_{M_i}$ of one of the
$M_i$'s is contained in the union of the support varieties of the other two.
\item 
If $G$ admits an embedding $G \hookrightarrow GL_N$ of exponential type defined over $\bF_p$, then 
$$V(G)_{M^{(1)}} \ = \ \{ \cE_{(B_0,B_1,B_2\ldots)} \in V(G) :
 \cE_{(B_1^{(1)},B_2^{(1)},\ldots)} \in V(G)_M\} .$$
\item  For any $r >0$, the restriction of $M$ to $kG_{(r)}$ is injective (equivalently, projective) 
if and only if the intersection of $V(G)_M$ with the subset 
$\{ \psi_{\ul B}: B_s = 0, s > r \}$ inside $V(G)$ equals $\{ \cE_{\ul 0} \}$.
\item  $V(G)_M \ \subset \ V(G)$ is a $G(k)$-stable subset.
\end{enumerate}
\end{thm}

\begin{proof}
Comparing Definitions 3.1 and 4.3, we see that to prove Property (1) it suffices to compare 
projectivity of $M$ when restricted along the actions
of   $\sum_{s \geq 0} (\cE_{B_s})_*(u_s)$  and of  $(\cE_{\ul \Lambda_r(B)})_*(u_r-1)$.
Since $M$ is assumed to have exponential degree $< p^r$, the action of 
$u = \sum_{s \geq 0} (\cE_{B_s})_*(u_s) $ on $M$ equals that of $\sum_{r=0}^{r-1}  (\cE_{B_s})_*(u_s)$ on $M$.
By \cite[4.6]{FP2}, Proposition \ref{sobaje-max} enables us to extend Proposition \ref{sobaje} to
arbitrary $kG_{a(r)}$-modules, thereby providing the required comparison of projectivity of $M$.

Property (1) immediately implies that $V(G)_M \subset V(G)$ is closed since $V_r(G)_M \subset V_r(G)$
is closed. 
As in the proof of Theorem \ref{Ga-items}, Properties (3), (4), and (5) are readily checked by checking 
at one $\cE_{\ul B} \in V(G)$ at a time.

Corollary \ref{func-frob} applies exactly as in 
the proof of Theorem \ref{Ga-items}
to prove property (6).  Similarly the proof of property (7) of Theorem \ref{Ga-items} applies
with minor notational changes (replacing the reference to Proposition \ref{com-Ga}
by a reference to Proposition \ref{sobaje} extended to modules of infinite dimension using
Proposition \ref{sobaje-max}) to prove Property (7).

Finally, to verify that $V(G)_M$ is
$G(k)$-stable, we observe that for any $x \in G(k)$ the rational $G$-module $M^x$ is isomorphic to
$M$.  Consequently,  the action of $G$ at $x \cdot \cE_{\ul B}$ on $M$ is isomorphic to the action 
of $G$ at $x \cdot \cE_{\ul B}$ on $M^x$  which equals the action of $G$ at $\cE_{\ul B}$ on $M$.
\end{proof}

\begin{remark}
\label{rem-poly}
A special case of Theorem \ref{G-items} is the case $G = GL_n$ and $M$ a polynomial $GL_n$-module
homogenous of some degree as in Example \ref{ex-poly}.  In particular, Theorem \ref{G-items} 
provides a theory of support varieties for modules over the Schur algebra $S(n,d)$ for $n \geq d$.
\end{remark}

For $M$ finite dimensional, the proof of Theorem \ref{G-items}(1) proves the following statement.

\begin{prop}
Let $G$ be a linear algebraic group equipped with a structure of 
exponential type, $M$  a finite dimensional rational $G$-module, and $r$
such that the exponential degree of $M$ is $< p^r$.  Then 
$$JT_{G,M}(\cE_{\ul B}) \ = \ JT_{G_{(r)},M}(\Lambda_r(\cE_{\ul B})).$$
\end{prop}

\begin{remark}
\label{rem-bound}
As observed in \cite[\S 3]{S1}, the condition on the upper bound for the exponential degree
in Theorem \ref{G-items}(1) can be weakened using \cite[Prop 8]{CLN}.  Namely,
the test for projectivity along the restriction of  (4.2.1) is equivalent
to the test for projectivity for the map $u \mapsto \sum_{s=0}^{r-1} (\cE_{B_s})_*(u_s)$
provided that the $(p-1)$-st power of $u_s$ acts trivially on $M$ for all $s \geq r$.
\end{remark}

Proposition \ref{sobaje-max} enables us to extend consideration of generalized support varieties
for $\bG_a$ (as defined in Definition \ref{def-ga-nonmax}) to linear algebraic groups equipped with
a structure of exponential type.

\begin{defn} 
\label{defn:gen}
Let $G$ be a linear algebraic group equipped with a structure of exponential type.
For any any $j, \ 1 \leq j < r$,  the non-maximal $j$-rank variety of a finite dimensional
rational $G$-module $M$
$$V^j(G)_M \ \subset \ V(G)$$
is defined as  the subset consisting of $\cE_{\ul 0}$ and
 those 1-parameter subgroups $\cE_{\ul B}: \bG_a \to G$ such that 
the rank of $(\sum_{s\geq 0} (\cE_{B_s})_*(u_s))^j: M \to M$ is not maximal.

As in \cite{FP2} and Definition \ref{def-ga-nonmax}, 
$$V^j(G_{(r)})_M \ \subset \ V_r(G)$$
is defined to be the subset of those $\cE_{\ul B}\circ i_r: G_{a(r)} \to \bG_a \to G$ such that either $\ul B = 0$
or the rank of $(\cE_{\ul B}\circ i_r)_*(u_{r-1}^j)$ is not .
\end{defn}

The following theorem is the extension of Theorem \ref{Ga-nonmax} to linear algebraic
groups equipped with a structure of exponential type.  

\begin{thm}
\label{G-nonmax}
Let $G$ be a linear algebraic group equipped with a structure of 
exponential type and let $M$ be a finite dimensional rational $G$-module with
nilpotent exponential degree $< p^r$.
For any $j, \ 1 \leq j < p^r$, the non-maximal $j$-rank variety $V^j(G)_M$ 
of a finite dimensional rational $G$-module
satisfies the following properties:
\begin{enumerate}
\item  $V^j(G)_M \ = \ pr_r^{-1}(\lambda_{r,r}(V^j_r(G)_M(k)))$.
\item
$V^j(G)_M  \subset V(G)$ is a proper, $G(k)$-stable,  closed subspace.
\item
$V^j(G)_M$ is a subspace of $V(G)_M$, with equality if and only if the action of
$G$ at some $\cE_{\ul B}$ has all Jordan blocks of size $p$.
\item  The restriction of $M$ to $kG_{(r)}$ is a module of constant 
$j$-rank if and only if the intersection of $V^j(G)_M$ with the subset 
$\{ \cE_{\ul B}: \ul B \not= 0, \ B_s = 0, s \geq r \} \ \subset \ V(G)$ is empty.
\end{enumerate}
\end{thm}

\begin{proof}
Exactly as remarked at beginning of the proof of Theorem \ref{Ga-nonmax}(1), 
Proposition \ref{sobaje-max} remains valid if the statement is modified by replacing 
Jordan type to $j$-rank thanks to \cite[3.5]{FP2}.  Thus the proof of Theorem \ref{G-items}
applies essentially verbatim to prove the first assertion by appealing to this modified 
version of Proposition \ref{sobaje-max}.

The fact that  $V^j(G)_M  \subset V(G)$ is closed follows from Property (1) and the fact proved in \cite[2.8]{FP2} that 
$V^j(G_{(r)})_M \subset V(G_{(r)})$ is closed; the fact that this inclusion is proper is tautological; the fact that 
it is $G(k)$-stable follows exactly as in the proof of Theorem \ref{G-items}(2).

Properties (3) and (4) are 
proved exactly as Theorem  \ref{Ga-nonmax}(3),(4).
\end{proof}


\section{Finite dimensional examples}
\label{sec:findim}

Let $G$ be a simple algebraic group of classical type and assume that $p > h$, where $h$
is the Coxeter number of $G$.   Let $\{ \alpha_1,\ldots\alpha_\ell\}$ be the set of simple roots
with respect to some Borel subgroup $B \subset G$,
$\{ \omega_1,\ldots,\omega_\ell\}$ be the set of fundamental
dominant weights of $G$, and for each $j$ write $\omega_j^\vee \ = \ \frac{2\omega_j}{\langle \alpha_j,\alpha_j \rangle}$.
By \cite[2.7]{F1}, the condition that 
all the high weights $\mu$ of $M$ of a rational $G$-module satisfy
\begin{equation}
\label{constraint1}
2\sum_{j=1}^l\langle \mu,\omega_j^\vee \rangle < p
\end{equation}
 implies that $u_i$ acts trivially on $M$
for $i \geq 1.$    Thus, the condition that every $u_i^{p-1}$ acts trivially on $M$ for $i \geq 1$
is implied by  the condition that
\begin{equation}
\label{constraint2}
2\sum_{j=1}^l\langle \mu,\omega_j^\vee \rangle < p(p-1).
\end{equation}

For a $p$-restricted dominant weight $\mu$, we denote by $I_\mu$ the a subset of the root lattice 
$\Pi$ determined by $\mu$ as in \cite[6.2.1]{NPV}) and denote by $\fu_{I_\mu}$ the the Lie algebra 
of the unipotent radical of the associated parabolic subgroup $P_{I_\mu}$.

 \begin{prop} (see \cite[3.1]{S1}
 \label{induced-tensor}
 Let $G$ be a simple algebraic group of classical type and assume that $p \geq  h$.
 Let $\mu_1, \ldots \mu_m$ be dominant weights of $G$, each
satisfying (\ref{constraint2}).   Denote by $M$ the tensor product of Frobenius twists of 
induced modules
$$M \ \equiv \ H^0(\mu_0) \otimes H^0(\mu_1)^{(1)} \otimes \cdots \otimes H^m(\mu_m)^{(m)}.$$
Then
\begin{equation}
\label{tensortwists}
V(G)_M \ = \  \{ \ul B:  \  B_i^{(i)} \in  G\cdot \fu_{I_{\mu_i}} \},
\end{equation}

\end{prop}

\begin{proof}
If $M$ is a finite dimensional rational $G$-module with the property that all high
weights $\mu$ of $M$ satisfy (\ref{constraint2}), then
Theorem \ref{G-items}(1) implies that $V(G)_M \ = \ pr_1^{-1}((V_1(G))(k))$.
In other words, $V(G)_M \ = \  \{exp_{\ul B}: B_0 \in V_1(G)_M \}$.

The ``Jantzen Conjecture" (see \cite[6.2.1]{NPV}) determines explicitly $V_1(G)_M \subset V(G)$ for 
$G = H^0(\mu)$ provided $p$ is good for $G$ (which is implied by $p \geq h$).  Namely, $V_1(G)_M$ 
is a single $G$ orbit, $V_1(G)_M \ = \ G\cdot\fu_{I_\mu}$,
Theorem \ref{G-items}(6) then enables us to to determine $V(G)_M$ for $M$ a Frobenius twist of $H^0(\mu)$
with $\mu$ satisfying (\ref{constraint2}).

Furthermore, Theorem \ref{G-items}(4) enables us to compute $V(G)_M$ for $M$ a 
tensor product of the form $H^0(\mu_0) \otimes H^0(\mu_1)^{(1)} \otimes \cdots \otimes H^m(\mu_m)^{(m)}$
provided that each $\mu_i$ satisfies satisfying (\ref{constraint2}), with answer given by 
(\ref{tensortwists}).
\end{proof}

\begin{ex}
For polynomial $GL_n$-modules, the  bounds for computations given in   Proposition \ref{induced-tensor}
can be weakened as
discussed in Example \ref{ex-poly}; namely, to insure that $u_r \in kGL_n$ acts trivially on the polynomial
$GL_n$-module $M$ of degree $d$, it suffices to assume that $p^ r >(p-1)d$.  Moreover, to insure that 
$u_r^{p-1}$ acts trivially on $M$, it suffices to assume that $p^r > d$.
\end{ex}

\begin{ex}
\label{ex-St}
Consider the polynomial $GL_2$ module $k[x,y]_{p^r-1} \ = St_r$.  Recall that $St_r$ is projective
as a $GL_{2(r)}$-module.  By Example \ref{ex-poly}, the action of $u_r$ on $St_r$ has 
$(p-1)$-st power trivial and the action of $u_{r+i}$ on $St_r$ is trivial for $ i > 0$.   Applying
\cite[4.3]{CLN}, we conclude that the action of $\sum_{s\geq 0}(exp_{A_s})_*(u_s)$ on $St_r$ is
free if some $A_s \not= 0, \ s < r$ and that the action is never free if $A_s = 0$ for all
$s < r$.  Consequently, $V(GL_2)_{St_r} \ = \ pr_r^{-1}(\{ \cE_0 \})$.
\end{ex}

\begin{ex}
\label{ex-max}
As in Remark \ref{rem-poly}, we consider a polynomial $GL_n$-module of degree $d$ with $(p-1)d < p^2$.
For such a rational $GL_n$-module $M$,  Example \ref{ex-poly} tells us that $u_r \in kGL_n$ acts trivially 
on $M$ for $r \geq 2$.  Thus, as in Theorem \ref{G-nonmax}(1), 
$V^j(GL_n)_M =  pr_2^{-1}(\lambda_{2,2}(V_2^j(GL_n)_M(k)))$.
The special case $V_2(SL_2)_M$ for $M$ an irreducible $SL_{2(2)}$-module is worked
out in detail in \cite[4.12]{FP3}.  
Consequently, this detailed computation leads to an explicit description of $V^j(GL_2)_M$ for $M$ 
an irreducible $SL_2$ module of weight $d$ satisfying $(p-1)d < p^2$ (i.e., $d \leq p+1$).
\end{ex}


\section{Infinite dimensional examples}
\label{sec:infdim}

In this section, we show that injective rational $G$-modules have trivial support variety (Proposition \ref{ker-inj})
and use this to compute non-trivial support varieties of certain special infinite dimensional rational $G$-modules
(Proposition \ref{prop:homog}).

\begin{prop}
\label{zero}
Let $G$ be a linear algebraic group equipped with a structure of exponential type and $I$ a 
rational $G$-module.  Then $V(G)_I \  = \ \{ \cE_0 \}$ if and only if $V_r(G)_I = \{ \cE_0 \}$ for all $r > 0$ 
if and only if the restriction of $I$ to each $G_{(r)}$ is injective.
\end{prop}

\begin{proof}
Using the observation following Definition \ref{str-exptype} that every infinitesimal 
1-parameter subgroup of $G$ lifts to a some $\cE_{\ul B}: \bG_a \to G$, we conclude that 
$\Lambda_r: V(G) \to V_r(G)(k)$ is surjective.  By construction, $V(G) \ \to \ \varprojlim_r V_r(G)(k)$
is an embedding.  Therefore, the 
definition of $V(G)_I \subset V(G)$ given in Definition \ref{supportG} implies  that
$V(G)_I = \{ \cE_0 \}$ if and only if $V_r(G)_I = \{ \cE_0 \}$ for all $r > 0$.   On the other hand, 
 for any $G_{(r)}$-module $M$ we know that  $V_r(G)_M = 0$ if and only $M$ is injective
as a  $G_{(r)}$-module $M$ (cf. \cite{P}).
\end{proof}

We apply Proposition \ref{zero} to conclude that the support of an injective rational $G$-module
(for $G$ of exponential type) is trivial.    The argument we present
is at least implicit in the work of Jantzen 
(cf. \cite[4.10-4.11]{J}) and was provided to us by J. Pevtsova.

\begin{prop}
\label{ker-inj}
Let $G$ be a linear algebraic group and $I$  an injective rational $G$-module.  
Then for each $r > 0$, the restriction of $I$ 
to $G_{(r)}$ is injective.

Consequently, if $G$ has a structure of exponential type and $I$ is an injective, rational $G$-module, then
$$V(G)_I \ = \ \{ \cE_0 \}.$$
\end{prop}

\begin{proof}
Any injective rational $G$-module $I$ is a summand of a direct sum of copies of the injective module $k[G]$,
so that it suffices to assume $I = k[G]$ because direct sums and  summands of  injective modules are injective.   
To prove that the restriction of $k[G]$ to $G_{(r)}$ is injective, it suffices
to prove for all finite dimensional $kG_{(r)}$-modules $M$ that $Ext^n_{G_{(r)}}(M,k[G]) = 0, \ n > 0$ .  For
such finite dimensional $M$, this is equivalent to showing that 
$$H^n(G_{(r)},M^\# \otimes k[G]) \ = \ R^n((-)^{G_{(r)}})(M^\# \otimes k[G]) = 0, \ n > 0.$$
Recall that the composition $= \ (-)^{G_{(r)}} \circ (-\otimes k[G])$ equals
$Ind_{G_{(r)}}^G(-)$ as functors from ($kG_{(r)}$-modules) to (rational $G$-modules)
Since $(-\otimes k[G])$ is exact, we conclude that 
$$R^nInd_{G_{(r)}}^G(-) \ = \ R^n(-)^{G_{(r)}} \circ (-\otimes k[G]), \ n \geq 0.$$
Since $G/G_{(r)}$ is affine, $G_{(r)} \subset G$ is exact as in \cite[I.5.13]{J}; thus, 
$R^nInd_{G_{(r)}}^G(-) = 0, n > 0$.  We thereby conclude that $H^n(G_{(r)},M^\# \otimes k[G]) = 0$
for $n > 0$ as required.

The second assertion follows immediately from the first and Proposition \ref{zero}.
\end{proof}

As a consequence of Proposition \ref{ker-inj}, we get the following additional computation.

\begin{ex}
\label{prop:homog}
Let $G \simeq H \rtimes K$ be an linear algebraic group equipped with a structure of exponential type
as in Example \ref{ex:sep}. Then 
$$V(G)_{k[K]} \ = \ V(H),$$ 
where $k[K] \ = \ k[G/H]$ is given the rational $G$-module structure 
obtained as the restriction along $\pi: G \to K$ of the natural rational $K$-structure on 
$k[K]$.
\end{ex}

\begin{proof}
Our hypothesis on $H \subset G$ implies  that  every 1-parameter subgroup $\psi: \bG_a \to H$ 
is of the form $\cE_{\ul B}$ with $\ul B \in \cC_r(\cN(\fh))$ for some $r > 0$.     Moreover, by
Example \ref{ex:sep}, $\pi$ is a map of groups of exponential type.

Observe that $\pi \circ \cE_{\ul B}$ is the trivial 1-parameter subgroup of $G/H$ if and only if $\ul B \in 
 \cC_r(\cN(\fg))$ maps to 0 in $ \cC_r(\cN(\fg/\fh))$ if and only if $\ul B \in  \cC_r(\cN(\fh))$.  In
 other words, $V(H) \subset V(G)$ consists of those 1-parameter subgroups $\cE_{\ul B}: \bG_a \to G$ 
 with the property that $\pi \circ \cE_{\ul B}: \bG_a \to G/H$ is trivial; of course, for each such $\cE_{\ul B}$,
 $\sum_{s \geq 0} (\cE_{B_s})_*(u_s)$ acts trivially on any rational $G$-module $M$.  In particular 
any such such $\cE_{\ul B}$ does not act freely of $k[G/H]$.  The proposition
 now follows by applying Proposition \ref{ker-inj} to $G/H$ which tells us that if  $\ul B$ maps to a non-zero element
 $\ul C \in \cC_r(\cN(\fg/\fh))$, then the action of $\sum_s (\cE_{C_s})_*(u_s)$ on $k[G/H]$ is free; on the
 other hand, the action of $\sum_s (\cE_{C_s})_*(u_s)$ equals the action of  $\sum_s (\cE_{B_s})_*(u_s)$
 on $k[G/H]$ since the $G$-action on $k[G/H]$ is that determined by $\pi: G \to G/H$.
 \end{proof}

\begin{ex}
\label{levi}
As in Example \ref{Sobaj}, let $G$ be a reductive group with $PSL_p$ not a factor of $[G,G]$
and assume that $p \geq h(G)$.  Let $P \subset G$ be a parabolic subgroup with unipotent 
radical $U \subset P$ and Levi factor group $\pi: P \to L = P/U$.  Then
\begin{equation}
V(P)_{k[L]} \ = \ V(U) \ \subset \ V(P).
\end{equation}
Here, $k[L]$ is given the structure of a rational $P$-module determined by extending the
usual action of $L$ on $k[L]$ along the quotient map $\pi: P \to L$.
\end{ex}
 
 Using the tensor product property of Theorem \ref{G-items} (4), we obtain the following examples.
 These examples are of interest for they suggest a means of realizing various subspaces of $V(G)$ 
 as the support variety of some (possibly infinite dimensional) rational $G$-module $M$.
 
 \begin{ex}
 Adopt the hypotheses and notation of Proposition \ref{induced-tensor}.  Let $\tilde M$ be given as 
 $$\tilde M \ \equiv \ .H^0(\mu_0) \otimes H^0(\mu_1)^{(1)} \otimes \cdots \otimes H^m(\mu_m)^{(m)} \otimes k[G]^{(m+1)}.$$
 Then 
 $$V(G)_{\tilde M} \ = \ \{ \ul B:   \  B_i^{(i)} \in  G\cdot \fu_{I_{\mu_i}}, 0 \leq i \leq m; \ B_j = 0, j > m \}.$$
 \end{ex}


\begin{thebibliography}{20}

\bibitem{Ca} J. Carlson, {\em The varieties and the cohomology ri§ ng of a module}, J. Algebra {\bf 85} (1983), 101-143.

\bibitem{CLN} J. Carlson, Z. Lin, and D. Nakano, {Support varieties for modules over Chevalley groups and
classical Lie algebras}, Trans. A.M.S. {\bf 360} (2008), 1870 - 1906.

\bibitem{CPS} E. Cline, B. Parshall, L. Scott, {\em Cohomology, hyperalgebras, and representations}, 
J.Algebra {\bf 63} (1980), 98-123.



\bibitem{F1} E. Friedlander, {\em Restrictions to $G(\bF_p)$ and $G_{(r)}$ of rational $G$-modules},
Compos. Math. {\bf 147} (2011), no. 6, 1955--1978.

\bibitem{F2} E. Friedlander, {\em Spectrum of group cohomology and support varieties}, J. K-Theory {\bf 11} (2013)
507-516.

\bibitem{FP1} E. Friedlander, J. Pevtsova, {\em $\Pi$-supports for modules 
for finite group schemes},  Duke. Math. J. {\bf 139} (2007), 317--368.

\bibitem{FP2} E. Friedlander, J. Pevtsova, {\em Generalized support varieties for finite group
schemes}, Documenta Math (2010), 197-222.

\bibitem{FP3} E. Friedlander, J. Pevtsova, {\em Constructions for infinitesimal group schemes},
Trans. A.M.S. {\bf 363} (2011), 6007-6061.

\bibitem{FPS} E. Friedlander, J. Pevtsova, A. Suslin, {\em Generic and  Jordan types},
Invent math {\bf 168} (2007), 485-522.

\bibitem{FS} E. Friedlander, A. Suslin, {Cohomology of finite group schemes over a field}, Invent. Math. 
{\bf 127}(1997), 209-270.

\bibitem{G} J. A. Green, {\em Polynomial represenations of $GL_n$}, Springer-Verlag Lecture Notes
in Math 830, 1980.

\bibitem{J} J.C. Jantzen, {\em Representations of Algebraic groups},
Academic Press, (1987).

\bibitem{McN} G. McNinch, {Optimal $SL(2)$ homomorphism}, Comment. Math. Helv. {\bf 80} (2005),
391-426.

\bibitem{NPV} D. Nakano, B.J. Parshall, D.C. Vella, {\em Support varieties for algebraic groups}, 
J. Reine Angew. Math. {\bf 547} (2002), 15-49.

\bibitem{O} V. Ostrik, {\em Cohomological supports for quantum groups}, Funct. Anal. Appl. {\bf 32} (1998),
237-246.

\bibitem{P} J. Pevtsova, {\em Infinite dimensional modules for Frobenius kernels}, J. Pure Appl. Algebra {\bf 173}
(2002), 59 - 86.

\bibitem{Q1} D. Quillen, {\em The spectrum of an equivariant cohomology ring: I}, Ann. Math {\bf 94}(1971), 549-572.

\bibitem{Q2} D. Quillen, {\em The spectrum of an equivariant cohomology ring: II}, Ann. Math {\bf 94}(1971), 573-602.


\bibitem{Sei} G. Seitz, {\em Unipotent elements, tilting modules, and saturation},
Invent. Math. {\bf 141} (2000), 467-502.


\bibitem{S1}  P. Sobaje, {\em Support varieties for Frobenius kernels of classical groups}. J. Pure and Appl. 
Algebra {\bf 216} (2013), 2657-2664.

\bibitem{S2} P. Sobaje,  {\em On exponentiation and infinitesimal one-parameter subgroups of reductive groups},
J. Algebra {\bf 385} (2013), 14-26.

\bibitem{SFB1} A. Suslin, E. Friedlander, C. Bendel,
{\em Infinitesimal 1-parameter subgroups and cohomology},
J. Amer. Math. Soc. {\bf 10} (1997), 693-728.

\bibitem{SFB2} A. Suslin, E. Friedlander, C. Bendel {\em Support
varieties for infinitesimal group schemes},  J. Amer. Math. Soc.
{10} (1997), 729-759.


\end{thebibliography}
\end{document}